\documentclass[a4paper, oneside]{amsart}
\usepackage{comment}
\usepackage[utf8]{inputenc}
\usepackage{ae,aecompl,amsbsy,amssymb,amsmath,amsthm,mathrsfs,mathtools,eurosym,amsfonts,bm,bbm,esint}
\usepackage[shortlabels]{enumitem}
\usepackage[colorlinks, citecolor = blue]{hyperref}

\usepackage{color}

\makeatletter
\DeclareFontFamily{OMX}{MnSymbolE}{}
\DeclareSymbolFont{MnLargeSymbols}{OMX}{MnSymbolE}{m}{n}
\SetSymbolFont{MnLargeSymbols}{bold}{OMX}{MnSymbolE}{b}{n}
\DeclareFontShape{OMX}{MnSymbolE}{m}{n}{
    <-6>  MnSymbolE5
   <6-7>  MnSymbolE6
   <7-8>  MnSymbolE7
   <8-9>  MnSymbolE8
   <9-10> MnSymbolE9
  <10-12> MnSymbolE10
  <12->   MnSymbolE12
}{}
\DeclareFontShape{OMX}{MnSymbolE}{b}{n}{
    <-6>  MnSymbolE-Bold5
   <6-7>  MnSymbolE-Bold6
   <7-8>  MnSymbolE-Bold7
   <8-9>  MnSymbolE-Bold8
   <9-10> MnSymbolE-Bold9
  <10-12> MnSymbolE-Bold10
  <12->   MnSymbolE-Bold12
}{}

\let\llangle\@undefined
\let\rrangle\@undefined
\DeclareMathDelimiter{\llangle}{\mathopen}%
                     {MnLargeSymbols}{'164}{MnLargeSymbols}{'164}
\DeclareMathDelimiter{\rrangle}{\mathclose}%
                     {MnLargeSymbols}{'171}{MnLargeSymbols}{'171}
\makeatother

\newcommand{\sums}[3]{\sum_{#1 = #2}^{#3}}
\newcommand{\R}{\mathbb{R}}

\newcommand{\N}{\mathbb{N}}
\newcommand{\avgL}{\textit{\L}}

\newcommand{\Z}{\mathbb{Z}}

\newcommand{\intD}{\,\mathrm{d}}
\newcommand{\setcolon}{\,\colon\,}

\newcommand\numberthis{\addtocounter{equation}{1}\tag{\theequation}}

\theoremstyle{plain}
\newtheorem{thm}{Theorem}[section]
\newtheorem{lem}[thm]{Lemma}
\newtheorem{prop}[thm]{Proposition}
\newtheorem{cor}[thm]{Corollary}

\theoremstyle{definition}
\newtheorem{dfntn}[thm]{Definition}

\theoremstyle{remark}
\newtheorem{rmrk}[thm]{Remark}

\makeatletter
\@namedef{subjclassname@2020}{%
  \textup{2020} Mathematics Subject Classification}
\makeatother

\address{Department of Mathematics and Statistics, P.O.B.~68 (Pietari Kalmin katu~5), FI-00014 University of Helsinki, Finland}
\email{aapo.laukkarinen@helsinki.fi}

\begin{document}
\title[Convex body domination for multi-scale operators]{Convex body domination for a class of multi-scale operators} 

\begin{abstract}
    The technique of sparse domination, i.e., dominating operators with sums of averages taken over sparsely distributed cubes, has seen rapid development recently within the realms of harmonic analysis.
    A useful extension of sparse domination called convex body domination allows one to estimate operators in matrix-weighted spaces. In this paper, we extend recent sparse domination results for a class of multi-scale operators due to Beltran, Roos and Seeger to the convex body setting and prove that this implies quantitative matrix-weighted norm bounds for these operators and their commutators.
\end{abstract}

\author{Aapo Laukkarinen}
\maketitle
 
\setcounter{tocdepth}{1}
\begingroup\hypersetup{linkcolor=black}
\tableofcontents
\endgroup

\section{Introduction}
A recent technique in harmonic analysis called sparse domination allows one to dominate many complicated operators by sums of averages taken over cubes that have major disjoint subsets. This type of domination, initiated by Lerner \cite{lerner_estimate_2013,lerner_simple_2013}, is primarily motivated by the sharp quantitative scalar-valued weighted norm inequalities that arise from it. In the vector case sparse domination walks into the problem that ordinary averages lose information about direction, which is often essential when considering matrix-valued weights. This problem can be addressed by replacing the averages with convex sets that naturally enjoy the property of containing information about behaviour in different directions. 

This technique of dominating via convex sets is called convex body domination and it was conjured by Nazarov et al. in \cite{nazarov_convex_2017} in order to extend the sharp weighted $A_2$ bound for Calderón-Zygmund operators (CZO) by Hytönen \cite{hytonen_sharp_2012} to the vector valued setting, where the weight functions are matrix-valued. While the conjectured linear bound on the dependence on the weight was not achieved, the method in \cite{nazarov_convex_2017} yielded a result that, at the time of writing, holds the record for most optimal. Consequent convex body domination results include two weight matrix $A_p$ bounds for CZOs by Cruz-Uribe et al. \cite{cruz-uribe_ofs_two_2018}, matrix weighted bounds for variational CZOs by Duong et al. \cite{duong_variation_2021}, for commutators of CZOs by Isralowitz et al. \cite{isralowitz_sharp_2021,isralowitz_commutators_2022} and for rough singular integral operators by Di Plinio et al. \cite{di_plinio_sparse_2021} and Muller and Rivera-Ríos \cite{muller_quantitative_2022}.  

In \cite{hytonen_remarks_2023} Hytönen generalized the notion of convex bodies to an arbitrary normed space $X$. As an application of this general setting one is able to perform convex body domination on the Bochner space $X=L^p(\R^n,E)$, where $E$ is an arbitrary Banach space. In this paper we will use this general framework to extend the sparse domination results of Beltran et al. \cite{beltran_multi-scale_2020} to the convex body setting. In particular, we consider a sum of operators
\[
    T=\sum_{j=N_1}^{N_2} T_j,
\]
where certain $2^j$-scalings of the operators $T_j$ enjoy uniform $L^p\to L^q$ boundedness and they are supported in a $1$-neighbourhood of the input function. The individual rescaled operators and their adjoints are also assumed to satisfy a certain regularity condition. Furthermore, the full sum enjoys uniform $L^p\to L^{p,\infty}$ and $L^{q,1}\to L^q$ bounds. Our main result is that a bilinear form related to $T$ permits convex body domination. For the exact statement see Theorem \ref{mainresult} and the preceding discussion in Section 3. We will also explore the properties of general operators that satisfy convex body domination. In particular, we prove new quantitative matrix-weighted bounds for these operators and their commutators.

The outline of the paper is as follows. In Section 2, we study the general framework for convex bodies from \cite{hytonen_remarks_2023} and equip ourselves with results that provide the foundation of our argument. In Section 3, we introduce the setting for multi-scale operators from \cite{beltran_multi-scale_2020} that is also the setting for our main result. This includes the dense subspace that we will work with in the following two sections. In Section 4, we begin the induction proof of the main result by proving the base case for the induction. In other words, we restrict ourselves to the single-scale setting and prove that the single-scale operators satisfy convex body domination. The induction step is provided in Section 5. In Section 6, we study matrix-weighted norm inequalities that arise from convex body domination. In particular, we will show that if $T$ satisfies convex body domination, then for $1\leq p<r<q\leq\infty$, $t=\frac{r}{p}$ and $s=\left(\frac{q}{r}\right)'$ we have
\[
    \|T\vec f\|_{L^r_{B_2^n}(W)}\leq C [W]_{A_t}^{\frac{t'}{r}+\frac{s}{r'}}[W]_{RH_{t,s}}^{\frac{1}{r}+\frac{s}{r'}} \|\vec f\|_{L^r_{B_1^n}(W)},
\]
where $[W]_{A_t}$ is the usual matrix-$A_t$ constant of $W$ and $[W]_{RH_{t,s}}$ is a natural extension of the reverse Hölder constant for matrix weights (see Definition \ref{RHtsclass}). 
Another application is explored in Section 7, where we study the commutators of operators that satisfy convex body domination. More explicitly, we will show that if $T$ satisfies convex body domination, then for $1\leq p<r<q\leq\infty$, $t=\frac{r}{p}$ and $s=\left(\frac{q}{r}\right)'$ we have
\begin{align*}
    \|[B,T]\vec f&\|_{L^r_{B_2^n}(W)}\leq\\& C\|B\|_{BMO}\left([W]_{A_t}^s[W]_{RH_{t,s}}^{s}+[W]_{A_t}^{\frac{1}{t-1}}\right)[W]_{A_t}^{\frac{t'}{r}+\frac{s}{r'}}[W]_{RH_{t,s}}^{\frac{1}{r}+\frac{s}{r'}}\|\vec f\|_{L^r_{B_1^n}(W)}.
\end{align*}
In Section 8, we prove that our results in the dense subspace extend to the global setting. Finally, in Section 9, following \cite{beltran_multi-scale_2020}, we give some examples of Fourier multiplier operators that enjoy a multi-scale structure and therefore satisfy convex body domination.

\section*{Acknowledgements}
The author would like to thank Tuomas Hytönen for supervision and insightful discussions that made this article possible. The author was supported by the Research Council of Finland through grants 346 314 (to Tuomas Hytönen) and 336 323 (to Timo Hänninen).

\section{General convex bodies}
We follow the framework from \cite{hytonen_remarks_2023}. Let  $\Bar{B}_X$ be the closed unit ball of a space $X$.
\begin{dfntn}
    For a real normed space $X$ and $\Vec{x}\coloneqq(x_i)_{i=1}^n\in X^n$, the convex body $\llangle\Vec{x}\rrangle_X$ of $\Vec{x}$ is 
    \[\llangle\Vec{x}\rrangle_X\coloneqq \{\langle\Vec{x},x^*\rangle\setcolon x^*\in \Bar{B}_{X^*} \}\subset\R^n,\]
    where $\langle\Vec{x},x^*\rangle=(\langle x_i,x^*\rangle)_{i=1}^n$.
\end{dfntn}
\noindent As the name suggests, the set $\llangle x \rrangle_X$ is convex, symmetric and compact (see Lemma 2.3 from \cite{hytonen_remarks_2023}). 
The Minkovski dot product of two sets $A,B\subset\R^n$ is defined by
\[A\cdot B\coloneqq \{a\cdot b\setcolon a\in A,b\in B\}.\]
In particular,
\begin{align*}
\llangle\Vec{x}\rrangle_X\cdot\llangle\Vec{y}\rrangle_Y &= \{\langle \Vec{x},x^*\rangle\cdot \langle \Vec{y},y^*\rangle\setcolon x^*\in \Bar{B}_{X^*}, y^*\in \Bar{B}_{Y^*}\}\\
&= \left\{\sums{i}{1}{n}\langle x_i,x^*\rangle \langle y_i,y^*\rangle\setcolon x^*\in \Bar{B}_{X^*}, y^*\in \Bar{B}_{Y^*}\right\}.
\end{align*}
Since the sets $\llangle\Vec{x}\rrangle_X$ and $\llangle\Vec{y}\rrangle_Y$ are convex bodies, the above set is a closed interval $[-c,c]$ and we can identify it with its right end point $c\in\R$. In particular,  for $r,s\in\R$ the inequalities $r\leq \llangle\Vec{x}\rrangle_X\cdot\llangle\Vec{y}\rrangle_Y\leq s$ mean $r\leq c\leq s$.

We extend the bilinear form $t\setcolon X\times Y\to\R$ to $X^n\times Y^n$ by 
\[t(\Vec{x},\Vec{y})=\sums{i}{1}{n}t(\Vec{x}\cdot\Vec{e}_i,\Vec{y}\cdot\Vec{e}_i),\numberthis\label{extendbiliform}\]
where $(\Vec{e}_i)_{i=1}^n$ is a fixed orthonormal basis of $\R^n$. In \cite{hytonen_remarks_2023} this definition is shown to be independent of the chosen orthonormal basis and for $\vec x\in X^n$ and $\vec y\in Y^n$ it holds that
\begin{align*}
    t(A\Vec{x},\vec y)=t(\Vec{x},A^\top \vec y).\numberthis\label{matrixswapbili}
\end{align*}

We will restate two results from \cite{hytonen_remarks_2023} that will be used later in this text.
\begin{lem}\label{TH1}
    Let $X,Y$ be normed spaces and $\Vec{f}\in X^n, \Vec{g}\in Y^n$. Let $\mathcal{E}_f$ be the John ellipsoid of $\llangle \Vec{f}\rrangle_X$ such that 
    \[
        \mathcal{E}_f\subset \llangle \Vec{f}\rrangle_X\subset \sqrt{n}\mathcal{E}_f,
    \]
    and suppose that $\mathcal{E}_f$ is non-degenerate (i.e. of full dimension). Let $R_f$ be the linear transformation such that $R_f\mathcal{E}_f=\Bar{B}_{\R^n}$ and let $(\Vec{e}_i)_{i=1}^n$be an orthonormal basis of $\R^n$. If 
    \[f_i\coloneqq R_f\Vec{f}\cdot\Vec{e}_i,\quad g_i\coloneqq R_f^{-\top}\Vec{g}\cdot\Vec{e}_i,\quad i=1,\dots,n,\]
    then
    \[
        \sums{i}{1}{n}\|f_i\|_X\|g_i\|_Y\leq n^\frac{3}{2}\llangle \Vec{f}\rrangle_X\cdot\llangle \Vec{g}\rrangle_Y.
    \]
\end{lem}
\begin{prop}\label{TH2}
Let $X,Y$ be normed spaces with subspaces $F\subset X$ and $G\subset Y$, and let $t\,\colon X\times Y\to \R$ be a bilinear form. Consider the following conditions 
\begin{enumerate}
    \item[$(1)$] For all $(f,g)\in F\times G$, we have
    \[|t(f,g)|\leq C\|f\|_X\|g\|_Y.\]
    \item[$(2)$] For all $(f,g)\in F\times G$, we have
    \[|t(\Vec{f},\Vec{g})|\leq C_n\llangle f\rrangle_X\cdot
\llangle g\rrangle_Y.\]
\end{enumerate}
For each $n\in\Z_+$, condition $(1)$ implies condition $(2)$ with $C_n=Cn^\frac{3}{2}$.
\end{prop}

Dealing with the whole dual space will sometimes turn out to be problematic when we want to calculate with the convex bodies. Fortunately, we can estimate the convex bodies by considering a norming subspace of the dual. To see this we need the so called Helly's condition.
\begin{prop}[Helly's condition]
    Let $E$ be a normed space. Let $e^*_i \in E^*$, $c_i \in \mathbb K$, $i = 1, \dots , n$, and $M > 0$ be given. The
following assertions are equivalent.
\begin{enumerate}\label{Helly}
    \item[(1)] For every  $\varepsilon> 0$ there exists $e \in E$ such that $\|e\|_E \leq M + \varepsilon$ and 
    \[
        \langle e,e^*_i\rangle=c_i,\quad\forall i=1,\dots,n.
    \]
    \item[(2)] We have
    \[
        \Big|\sums{i}{1}{n}a_ic_i\Big|\leq M\Big\|\sums{i}{1}{n}a_ie_i^*\Big\|_{E^*},\quad \forall a_i\in\mathbb K,\,\,i=1,\dots,n.
    \]
\end{enumerate}
\end{prop}
\noindent See, e.g., \cite[Proposition 4.3.1]{hytonen_analysis_2016}  for a proof of Helly's condition.
\begin{lem}\label{convBodyCalc}
    Let $Z\subset X^*$ be norming for $X$. For $\vec x \in X^n$ we have
    \[
        \llangle \vec x\rrangle_X=\bigcap_{\varepsilon>0}(1+\varepsilon)\{\langle \vec x, z\rangle\,\colon\, z\in \Bar{B}_Z\}.
    \]
\end{lem}
\begin{proof}
    Let $J\colon X\to Z^*$ be the mapping defined by $\langle z,Jx\rangle=\langle x,z\rangle$ for $x\in X$ and $z\in Z$. Since $Z$ is norming for $X$, the mapping $J$ is an isometry. Indeed,
    \[
        \|Jx\|_{Z^*}=\sup_{\|z\|_Z\leq1}\langle z,Jx\rangle=\sup_{\|z\|_Z\leq1}\langle x,z\rangle=\|x\|_X.
    \]
    We will now check that condition $(\mathit{2})$ of Proposition \ref{Helly} 
    holds with $E=Z$, $c_i=\langle x_i,x^*\rangle$, $e_i^*= Jx_i$ and $M=\|x^*\|_{X^*}$. Since $J$ is an isometry, we have
    \begin{align*}
        \Big|\sums{i}{1}{n}a_i\langle x_i,x^*\rangle\Big|=\Big|\langle\sums{i}{1}{n}a_ix_i,x^*\rangle\Big|&\leq \|x^*\|_{X^*}\Big\|\sums{i}{1}{n}a_ix_i\Big\|_X\\&=\|x^*\|_{X^*}\Big\|\sums{i}{1}{n}a_iJx_i\Big\|_{Z^*}.
    \end{align*}
    Now Helly's condition gives us that for every $\varepsilon>0$ there exists $z_\varepsilon\in Z$ such that $\|z_\varepsilon\|_Z\leq \|x^*\|_{X^*}+\varepsilon$ and 
    \[
        \langle x_i,x^*\rangle=\langle z_\varepsilon,Jx_i\rangle=\langle x_i,z_\varepsilon\rangle,\qquad \forall i=1,\dots,n.
    \]
    Thus we have 
    \begin{align*}
        \llangle\vec x\rrangle_X&=\{\langle\vec x,x^*\rangle\,\colon\,\|x^*\|_{X^*}\leq 1\}\\&\subset\bigcap_{\varepsilon>0}\{\langle\vec x,z_\varepsilon\rangle\,\colon\,\|z_\varepsilon\|_{Z}\leq 1+\varepsilon\}\\&=\bigcap_{\varepsilon>0}\{\langle\vec x,z_\varepsilon\rangle\,\colon\,\Big\|\frac{z_\varepsilon}{1+\varepsilon}\Big\|_{Z}\leq 1\}\\&=\bigcap_{\varepsilon>0}\{\langle\vec x,(1+\varepsilon)z_\varepsilon\rangle\,\colon\,\|z_\varepsilon\|_{Z}\leq 1\}\\&=\bigcap_{\varepsilon>0}(1+\varepsilon)\{\langle\vec x,z_\varepsilon\rangle\,\colon\,\|z_\varepsilon\|_{Z}\leq 1\}.
    \end{align*}
    Conversely, we have
    \begin{align*}
        \bigcap_{\varepsilon>0}(1+\varepsilon)\{\langle\vec x,z_\varepsilon\rangle\,\colon\,\|z_\varepsilon\|_{Z}\leq 1\}&\subset\bigcap_{\varepsilon>0}(1+\varepsilon)\{\langle\vec x,x^*\rangle\,\colon\,\|x^*\|_{X^*}\leq 1\}\\&=\bigcap_{\varepsilon>0}(1+\varepsilon)\llangle \vec x\rrangle_{X},
    \end{align*}
    and by the properties of convex bodies we get
    \[
        \bigcap_{\varepsilon>0}(1+\varepsilon)\llangle \vec x\rrangle_{X}=\llangle \vec x\rrangle_{X}.
    \]
    Indeed, if $a\in\bigcap_{\varepsilon>0}(1+\varepsilon)\llangle \vec x\rrangle_{X}$, then for any $\varepsilon>0$ we have 
    \[
        a_\varepsilon\coloneqq\frac{a}{1+\varepsilon}\in\llangle \vec x\rrangle_{X}.
    \]
    The sequence $a_\varepsilon$ clearly converges to $a$ as $\varepsilon\to0$. On the other hand, by compactness we can extract a subsequence $a_{\varepsilon_j}$ that converges to some $\tilde{a}\in\llangle \vec x\rrangle_{X}$ and it follows that $a=\tilde a\in\llangle\vec x\rrangle_X$. The reverse inclusion $\llangle \vec x\rrangle_{X}\subset\bigcap_{\varepsilon>0}(1+\varepsilon)\llangle \vec x\rrangle_{X}$ is clear, since $\llangle \vec x\rrangle_{X}$ is convex and symmetric.
\end{proof}

\section{Setting and the main result}
We follow the notation and set-up from \cite{beltran_multi-scale_2020}. We denote the collection of dyadic subcubes of a cube $Q$ by $\mathscr{D}(Q)$. We will work with a dyadic lattice $\mathscr D$ that satisfies the following four properties:
\begin{enumerate}
    \item The cubes in $\mathscr D$ have side length of the form $2^k$ with $k\in \Z$.
    \item If $Q\in\mathscr D$ and $Q'\in \mathscr D(Q)$, then $Q'\in \mathscr D$.
    \item If $Q',Q''\in\mathscr D$, then there exists $Q\in\mathscr{D}$ such that $Q',Q''\in\mathscr{D}(Q)$.
    \item Every compact set in $\R^d$ is contained in some cube from $\mathscr D$.
\end{enumerate}
We say that a collection of cubes $\mathscr{Q}\subset\mathscr D$ is $\gamma$-sparse if for every cube $Q\in \mathscr{Q}$ there exists a subset $E_Q$ such that $|E_Q|\geq \gamma|Q|$ and the subsets $E_Q$ are pairwise disjoint.

For a Banach space $B$ and $p, r \in
[1,\infty]$ we define the Lorentz  space $L^{p,r}_B$
as the space of strongly measurable
functions $f \colon \R^d \to B$ so that the function $x \mapsto |f(x)|_B$ is in the scalar
Lorentz space $L^{p,r}$. The Bochner space $L^p_B$ is defined  analogously.

We also define $S_B$ to be the space of all functions of the form
\[
    f=\sums{i}{1}{N}a_i\mathbbm1_{Q_i},
\]
where $a_i\in B$ and $Q_i$ are dyadic cubes of $\R^d$ that are contained in a compact set.
For Banach spaces $B_1, B_2$ we consider the
space $\operatorname{Op}_{B_1,B_2}$
of linear operators $T$ mapping functions in $S_{B_1}$
to weakly
measurable $B_2$-valued functions  with the property that $x \mapsto \langle T f(x), \lambda\rangle$ is locally
integrable for any $\lambda \in B_2^*$. 

For a function $f$ we define $\operatorname{Dil}_tf(x) = f(tx)$ and for an operator $T$ we define the
dilated operator $\operatorname{Dil}_tT$ by
\[
\operatorname{Dil}_tT=\operatorname{Dil}_t\circ\hspace{0.1cm} T\circ\operatorname{Dil}_{t^{-1}}.
\]
There will be six conditions that we assume the multi-scale operator to satisfy.
\begin{dfntn}\label{BRSop}
    We say that the sum $T=\sum_jT_j$ is a BRS operator (named after the authors of \cite{beltran_multi-scale_2020}, where this class of operators is studied) if it satisfies the following conditions.

    The single scale operators $T_j$ are in $\operatorname{Op}_{B_1,B_2}$ such that the support condition \[\operatorname{supp} (\operatorname{Dil}_{2^j}T_j)f\subset \{x\in\R^n\setcolon \operatorname{dist}(x,\operatorname{supp}f)\leq1\}\tag{T1}\label{supportCond}\]
holds for all $f\in S_{B_1}$, and they satisfy the bound
\[\underset{j\in\Z}{\sup}\,||\operatorname{Dil}_{2^j}T_j||_{L^p_{B_1}\to L^q_{B_2}}\leq A_\circ(p,q).\tag{T2}\label{singleScaleLpqbound}\]
Furthermore, the single scale operators and their adjoints satisfy the following regularity conditions 
\begin{align*}\label{singleScalereqcond}
    \sup_{|h|\leq1}|h|^{-\kappa}\sup_{j\in\Z}\|(\operatorname{Dil}_{2^j}T_j)\circ\Delta_h\|_{L^{p}_{B_1}\to L^{q}_{B_2}}\leq B \tag{T3}
\end{align*}
and
\begin{align*}\label{adjSingleScalereqcond}
    \sup_{|h|\leq1}|h|^{-\kappa}\sup_{j\in\Z}\|(\operatorname{Dil}_{2^j}T^*_j)\circ\Delta_h\|_{L^{q'}_{B_2^*}\to L^{p'}_{B_1^*}}\leq B,\tag{T4}
\end{align*}
where $\Delta_h f(x)=f(x+h)-f(x)$ and $\kappa>0$ is a fixed parameter.

Let $N_1,N_2$ be integers such that $N_1\leq N_2$. We consider the sum $\sums{j}{N_1}{N_2}T_j$ and stipulate that it satisfies 
\begin{align*}\label{weakLpbound}
    \sup_{N_1\leq N_2}\left\|\sums{j}{N_1}{N_2}T_j\right\|_{L^p_{B_1}\to L^{p,\infty}_{B_2}}\leq A(p)\tag{T5}
\end{align*}
and
\begin{align*}\label{restStrongLqbound}
    \sup_{N_1\leq N_2}\left\|\sums{j}{N_1}{N_2}T_j\right\|_{L^{q,1}_{B_1}\to L^{q}_{B_2}}\leq A(q).\tag{T6}
\end{align*}
\end{dfntn}

To facilitate the discussion for convex body domination we use the following general definition.
\begin{dfntn}
    Suppose that pairs of normed spaces $(X(Q),Y(Q))$ are associated to every dyadic cube $Q \in \mathscr D$. We say that a bilinear form $t\colon F\times G\to\R$ satisfies $(X,Y)$ convex body domination, if $F\subset X(Q)$ and $G\subset Y(Q)$ for every $Q\in \mathscr{D}$, and if there exists a constants $C>0$ and $\gamma>0$ such that for every $(\vec f,\vec g)\in F^n\times G^n$, there exists a $\gamma$-sparse collection $\mathscr{Q}\subset\mathscr{D}$ that satisfies 
    \[
        |t(\vec f, \vec g)|\leq C \sum_{Q\in\mathscr{Q}} |Q|\llangle \vec f\rrangle_{X(Q)}\cdot\llangle \vec g\rrangle_{Y(Q)}.
    \]
    In addition, we say that a linear operator $T$ satisfies $(X,Y)$ convex body domination, if the associated bilinear form $t(\vec f,\vec g)\coloneqq\langle T\vec f,\vec g\,\rangle$ satisfies $(X,Y)$ convex body domination. 
\end{dfntn}
\noindent Our results are in the case $(X,Y)=(\textit{\L}^p_{B_1},\textit{\L}^{q'}_{B_2^*})$, where 
\[
    \|f\|_{\textit{\L}^p(Q,B_1)}\coloneqq \left(\fint_Q|f(x)|_{B_1}^p\intD x\right)^\frac{1}{p}
\]
and $\|\cdot\|_{\textit{\L}^{q'}(Q,B_2^*)}$ is defined analogously. We fix an orthonormal basis $(\vec e_i)_{i=1}^{n}$ of $\R^n$ and denote $f_i\coloneqq \vec f\cdot\vec e_i$. Then the bilinear form is defined by \[\langle T\vec{f},\vec g\,\rangle\coloneqq\sums{i}{1}{n}\langle T
f_i,g_i\rangle\coloneqq\sum_{i=1}^{n}\int_{\R^n}\langle T f_i(x),g_i(x)\rangle_{(B_2,B_2^*)}\intD x.\]

The main result of this article is the following theorem, which asserts that the multi-scale operator $\sum_{j}T_j$ satisfies $(\textit{\L}^p_{B_1},\textit{\L}^{q'}_{B_2^*})$ convex body domination.
\begin{thm}\label{mainresult} Let $T=\sum_{j=N_1}^{N_2}T_j$ be a BRS operator and let $N_1,N_2$ be integers such that $N_1\leq N_2$. Let $1< p\leq q<\infty$ and $0<\gamma<1$.
For all $\vec{f}\in S_{B_1}^n$ and $\vec{g}\in S_{B_2^*}^n$ there exists a $\gamma$-sparse collection  of dyadic cubes $\mathscr{Q}$ such that
\begin{align*}
    |\langle T\Vec{f},\vec g\,\rangle|\lesssim_{d,p,q,\gamma,\kappa} \mathcal{C}n^{\frac{3}{2}+\frac{1}{p}+\frac{1}{q'}}\sum_{Q\in\mathscr{Q}} |Q| \,\llangle\Vec{f}\rrangle_{\textit{\L}^p(Q,B_1)}\cdot \llangle\Vec{g}\rrangle_{\textit{\L}^{q'}(Q,B_2^*)},
\end{align*}
where $\mathcal{C}=A(p)+A(q)+A_\circ(p,q)\log(2+\frac{B}{A_\circ(p,q)})$.
\end{thm}
The proof will be done by adapting the induction argument of \cite{beltran_multi-scale_2020} on the difference $N_2-N_1$. The base case is covered in Section 4 and the rest of the proof in Section 5.

\section{Single scale estimate}
We use the notation $f_i\coloneqq\vec f\cdot\vec e_i$ throughout this section, where, where $(\vec e_i)_{i=1}^n$ is a fixed orthonormal basis of $\R^n$. For this section we will only need conditions \ref{supportCond} and \ref{singleScaleLpqbound} from Definition \ref{BRSop}.
Note that by rescaling \ref{singleScaleLpqbound} we get 
\[
    ||T_j||_{L^p_{B_1}\to L^q_{B_2}}\leq 2^{-jd(\frac{1}{p}-\frac{1}{q})}A_\circ(p,q).\numberthis\label{rescaledSingleScaleBound}
\]
Indeed, by defining $\tilde{f}\coloneqq \operatorname{Dil}_{2^j}f$ we get
\[
    \|T_jf\|_{ L^{q}_{B_2}}=\|T_j(\operatorname{Dil}_{2^{-j}}\tilde{f})\|_{ L^{q}_{B_2}}
\]
and via change of variables $x=2^jy$, we get
\begin{align*}
    \left(\int|T_j(\operatorname{Dil}_{2^{-j}}\tilde{f})(x)|_{B_2}^q\intD x\right)^\frac{1}{q} &= 2^{\frac{jd}{q}} \left(\int|(\operatorname{Dil}_{2^{j}}T_j)\tilde{f}(y)|_{B_2}^q\intD y\right)^\frac{1}{q}
    \\&\leq 2^{\frac{jd}{q}}A_\circ(p,q) \left(\int|\tilde{f}(y)|_{B_1}^p\intD y\right)^\frac{1}{p}
    \\&= 2^{\frac{jd}{q}}A_\circ(p,q) \left(\int|f(2^jy)|_{B_1}^p\intD y\right)^\frac{1}{p}.
\end{align*}
Changing variables back leads to \eqref{rescaledSingleScaleBound}.

The following proposition will serve as the base case in the induction proof of the $(\textit{\L}^p_{B_1},\textit{\L}^{q'}_{B_2^*})$ convex body domination for the sum $\sum_j T_j$. The conclusion is that the single scale operators $T_j$ satisfy $(\textit{\L}^p_{B_1},\textit{\L}^{q'}_{B_2^*})$ convex body domination.
\begin{prop}\label{singleScaleDomination}
    Let $T_j$ satisfy \ref{supportCond} and \ref{singleScaleLpqbound} from Definition \ref{BRSop}. Then for all $\vec{f}\in S_{B_1^n}$ and $\vec{g}\in S_{(B_2^*)^n} $ there exists a disjoint family $\mathscr{Q}$ of dyadic cubes such that 
\[
    |\langle T_j\vec{f},\vec g\,\rangle|\leq 3^{d(\frac{1}{q'}-\frac{1}{p'})}A_\circ(p,q)\,n^\frac{3}{2} \sum_{Q\in\mathscr{Q}}|3Q|\,\llangle \Vec{f}\rrangle_{\textit{\L}^p(3Q,B_1)}\cdot\llangle\Vec{g}\rrangle_{\textit{\L}^{q'}(3Q,B_2^*)}.
\]
\end{prop}
\begin{proof}
Consider 
\begin{align*}    &a\coloneqq\langle T_j\vec{f},\vec{g}\,\rangle,\qquad a_Q\coloneqq t_Q(\vec{f},\vec{g})\coloneqq \Big\langle T_j\big[\mathbbm 1_Q \vec f\,\,\big],\mathbbm 1_{3Q}\vec g\,\Big\rangle
\end{align*}
and 
\[c_Q=C_n\, |3Q|\,\llangle \Vec{f}\rrangle_{\textit{\L}^p(3Q,B_1)}\cdot\llangle\Vec{g}\rrangle_{\textit{\L}^{q'}(3Q,B_2^*)},\] where $Q$ is a cube and $C_n=3^{d(\frac{1}{q'}-\frac{1}{p'})}A_\circ(p,q)\,n^{\frac{3}{2}}$.
We  we will check that there exists a disjoint family of cubes $\mathscr{Q}$ such that 
\begin{align}\label{asum}              a=\sum_{Q\in\mathscr{Q}}a_Q
\end{align}
and for every $Q\in\mathscr{Q}$, we have
\begin{align}\label{aQbound}
    |a_Q|\leq c_Q.
\end{align}

\noindent Indeed, we can choose $\mathscr{Q}$ to be a partition of $\R^d$ consisting of cubes that have side length equal to $2^j$.
\begin{enumerate}
    \item[\eqref{asum}]  Note that for $Q\in\mathscr{Q}$ the support condition gives that $T_j[\mathbbm{1}_Qf_i]$ is supported in $3Q$. Now we can calculate 
    \begin{align*}
        a&= \sums{i}{1}{n}\langle T_jf_i,g_i\rangle=
        \sums{i}{1}{n}\left\langle T_j\Big[\sum_{Q\in\mathscr{Q}}\mathbbm{1}_Qf_i\Big],g_i\right\rangle\\&=\sums{i}{1}{n}\sum_{Q\in\mathscr{Q}}\langle T_j[\mathbbm{1}_Qf_i],g_i\rangle=\sum_{Q\in\mathscr{Q}}\sums{i}{1}{n}\langle T_j[\mathbbm{1}_Qf_i],\mathbbm{1}_{3Q}g_i\rangle=\sum_{Q\in\mathscr{Q}}a_Q.
    \end{align*}
    Since $f_i$ is compactly supported, there are no issues of convergence in the above calculation.
    \item[\eqref{aQbound}] Note that the scalar Hölder inequality applied to the function \[x\mapsto|T_j[f\mathbbm1_Q](x)|_{B_2}|g\mathbbm1_{3Q}(x)|_{B_2^*}\] implies that
    \begin{align*}
        \langle T_j[f\mathbbm1_Q],g\mathbbm1_{3Q}\rangle&=\int_{\R^n}\langle T_j[f\mathbbm1_Q](x),g\mathbbm1_{3Q}(x)\rangle_{(B_2,B_2^*)}\intD x\\&\leq \int_{\R^n} |T_j[f\mathbbm1_Q](x)|_{B_2}|g\mathbbm1_{3Q}(x)|_{B_2^*}\intD x\\&\leq \|T_j[f\mathbbm1_Q]\|_{L^p_{B_2}}\|g\mathbbm1_{3Q}\|_{L^{p'}_{B_2^*}}.
    \end{align*} 
    Applying this and \eqref{rescaledSingleScaleBound} we get  
    \begin{align*}
        |t_{Q}(f,g)|& =|\langle T_j[f\mathbbm{1}_{Q}],g\mathbbm{1}_{3Q}\rangle|\\&\leq \|T_j[f\mathbbm1_Q]\|_{L^q_{B_2}}\|g\mathbbm1_{3Q}\|_{L^{q'}_{B_2^*}}\\&\leq A_\circ(p,q)2^{-jd(\frac{1}{p}-\frac{1}{q})}\|f\mathbbm1_Q\|_{L^p_{B_1}}\|g\mathbbm1_{3Q}\|_{L^{q'}_{B_2^*}}\\&=3^\frac{d}{q'}A_\circ(p,q)\,|Q|\,\|f\|_{\textit{\L}^p(Q,B_1)}\|g\|_{\textit{\L}^{q'}(3Q,B_2^*)}\\
        &\leq3^{d\left(\frac{1}{q'}-\frac{1}{p'}\right)}A_\circ(p,q)\,|3Q|\,\|f\|_{\textit{\L}^p(3Q,B_1)}\|g\|_{\textit{\L}^{q'}(3Q,B_2^*)},
    \end{align*}
    for all $f\in S_{B_1}$ and $g\in S_{B_2^*}$.
    This is the first condition from Proposition \ref{TH2} with \[t=t_Q,\quad  X=\textit{\L}^p(3Q,B_1), \quad Y=\textit{\L}^{q'}(3Q,B_2^*)\] and \[C=3^{d\left(\frac{1}{q'}-\frac{1}{p'}\right)}A_\circ(p,q)\,|3Q|.\] 
    By Proposition \ref{TH2} this implies that
    \[|a_Q|=|t_Q(\Vec{f},\Vec{g})|\leq C_n\,|3Q|\, \llangle \Vec{f}\rrangle_{\textit{\L}^p(3Q,B_1)}\cdot\llangle\Vec{g}\rrangle_{\textit{\L}^{q'}(3Q,B_2^*)}=c_Q.\]
\end{enumerate}
Now we can conclude
\begin{align*}
    |\langle T_j\vec{f},\vec{g}\,\rangle|=|a| &= |\sum_{Q\in\mathscr{Q}}a_Q|\leq \sum_{Q\in\mathscr{Q}}|a_Q|\leq \sum_{Q\in\mathscr{Q}}c_Q\\&=C_n\sum_{Q\in\mathscr{Q}}|3Q|\,\llangle \Vec{f}\rrangle_{\textit{\L}^p(3Q,B_1)}\cdot\llangle\Vec{g}\rrangle_{\textit{\L}^{q'}(3Q,B_2^*)}\\&=3^{d(\frac{1}{q'}-\frac{1}{p'})}A_\circ(p,q)\,n^{\frac{3}{2}}\sum_{Q\in\mathscr{Q}}|3Q|\,\llangle \Vec{f}\rrangle_{\textit{\L}^p(3Q,B_1)}\cdot\llangle\Vec{g}\rrangle_{\textit{\L}^{q'}(3Q,B_2^*)}.
\end{align*}
\end{proof}

\section{Estimating the full sum}
In this section we will prove Theorem \ref{mainresult}. To facilitate the proof we will need the following definitions.
\begin{dfntn}
    Given a dyadic cube $Q_0$ let 
    \[\mathfrak{G}_{Q_0}(\Vec{f},\Vec{g})=\sup\sum_{Q\in\mathfrak{S}} |3Q|\,\llangle \Vec{f}\rrangle_{\textit{\L}^p(3Q,B_1)}\cdot \llangle\Vec{g}\rrangle_{\textit{\L}^{q'}(3Q,B_2^*)},\]
    where the supremum is taken over all $\gamma$-sparse collections $\mathfrak{S}$ consisting of cubes in $\mathscr{D}(Q_0)$.
\end{dfntn}
\begin{dfntn}
    For $N\in \N$, let $U(N)$ be the smallest constant so that for all families of operators $\{T_j\}$ satisfying the assumptions \ref{supportCond}--\ref{restStrongLqbound} of Definition \ref{BRSop} for all pairs $(N_1,N_2)$ with $N_2-N_1\leq N$ and for all dyadic cubes $Q_0$ of side length $2^{N_2}$, we have
    \[|\langle T\Vec{f},\vec g\,\rangle|\leq U(N)\mathfrak{G}_{Q_0}(\Vec{f},\Vec{g})\numberthis\label{inductDefIneq}\]
    whenever $\vec f\in S_{B_1^n}$ with $\operatorname{supp}(\vec f)\subset Q_0^n$ and $\vec g\in S_{(B_2^n)^*}$.
\end{dfntn}
\noindent Note that, as a consequence of the single scale support condition \ref{supportCond}, in the above definition we can assume that  $\operatorname{supp}(\vec g)\subset 3Q_0^n$.
\begin{rmrk}\label{DegenToGen}
    By similar arguments as in \cite[Proof of Proposition 4.2]{hytonen_remarks_2023}, one can show that when pursuing inequalities that are of the form \eqref{inductDefIneq}, 
    we may assume that the John ellipsoid $\mathcal{E}_f$ of $\llangle f\rrangle_{\textit{\L}^p(3Q_0,B_1)}$ is non-degenerate. 
    
    Indeed, let $\mathcal{O}$ denote the orthogonal projection of $\R^
    n$ onto the linear span of $\llangle f\rrangle_{\textit{\L}^p(3Q_0,B_1)}$ that we denote by $H$. Now since $(\Vec{f},\Vec{g})\mapsto\langle T\Vec{f},\vec g\,\rangle$ defines a bilinear form, on the left-hand side we have
    \[
        \langle T\vec f,\vec g\rangle=\langle T\left[\mathcal{O}^\top\mathcal{O}\vec f\right],\vec g\rangle=\langle T\left[\mathcal{O}\vec f\right],\mathcal{O}\vec g\rangle.
    \]
    For the right-hand side we fix a $\gamma$-sparse collection $\mathfrak S\subset \mathscr D(Q_0)$ and calculate for $Q\in\mathfrak S$ that
    \begin{align*}
        \llangle \Vec{f}\rrangle_{\textit{\L}^p(3Q,B_1)}\cdot \llangle\Vec{g}\rrangle_{\textit{\L}^{q'}(3Q,B_2^*)}&=\llangle \mathcal{O}^\top\mathcal{O}\Vec{f}\rrangle_{\textit{\L}^p(3Q,B_1)}\cdot \llangle\Vec{g}\rrangle_{\textit{\L}^{q'}(3Q,B_2^*)} \\&= \llangle \mathcal{O}\Vec{f}\rrangle_{\textit{\L}^p(3Q,B_1)}\cdot \llangle\mathcal{O}\Vec{g}\rrangle_{\textit{\L}^{q'}(3Q,B_2^*)}.
    \end{align*}
    By summing over $Q\in\mathfrak S$ and taking supremums over $\gamma$-sparse collections in $\mathscr D(Q_0)$ we get \[\mathfrak G_{Q_0}(\vec f,\vec g)=\mathfrak G_{Q_0}(\mathcal{O}\vec f,\mathcal{O}\vec g).\]
    Thus we can assume that the functions $\vec f,\vec g$ are in $H$ and $\mathcal{E}_f$ is non-degenerate as a subset of $H$. 
\end{rmrk}

Our next goal is to show that
\begin{align*}
    U(N)\lesssim_{d,p,q,\gamma,\kappa} \mathcal{C}n^{\frac{3}{2}+\frac{1}{p}+\frac{1}{q'}}
\end{align*}
uniformly in $N\in\N$. The case $n=1$ is due to \cite{beltran_multi-scale_2020} and we adapt their argument in our setting. The proof will be done by induction on $N$ and Proposition \ref{singleScaleDomination} gives us the base case
\[U(0)\leq 3^{d(\frac{1}{q'}-\frac{1}{p'})}A_\circ(p,q)\,n^\frac{3}{2}.\]
The inductive step is included in the following lemma. Note that after the proof of the following lemma we are done since the rest of the proof of theorem \ref{mainresult} is similar to Section 4.2 from \cite{beltran_multi-scale_2020}.
\begin{lem}\label{inductiveLemma}
There is a constant $c=c_{d,p,q,\gamma,\kappa}$ such that for all $N>0$,
\[U(N)\leq \max\{U(N-1),\,c\,\mathcal{C}n^{\frac{3}{2}+\frac{1}{p}+\frac{1}{q'}}\}.\]
\end{lem}
\begin{proof}
Let $Q_0$ be a dyadic cube of side length $2^{N_2}$ and  $\mathcal{M}_pf=(\mathcal{M}|f|^p)^{1/p}$, where $\mathcal{M}$ is the Hardy-Littlewood maximal operator. 
By Remark \ref{DegenToGen}, we may assume that the John ellipsoid $\mathcal{E}_f$ of the convex body $\llangle \vec f\rrangle_{\textit{\L}^p(Q_0,B_1)}$ is non-degenerate and consider the linear transformation $R_f$ that maps $\mathcal{E}_f$ to the closed unit ball of $\R^n$. For the rest of the proof we denote $f_i\coloneqq R_f\vec f\cdot \vec e_i$ and $g_i\coloneqq R_f^{-\top}\vec g\cdot \vec e_i$. Note that we have
\[
    \langle T\Vec{f},\vec g\,\rangle=\left\langle T\Big[R_f^{-1}R_f\Vec{f}\Big],\Vec{g}\right\rangle=\left\langle T\Big[R_f\Vec{f}\Big],R_f^{-\top}\Vec{g}\right\rangle=\sums{i}{1}{n}\langle Tf_i,g_i\rangle.
\]

Then we define $\Omega_i\coloneqq\Omega_i^{(f)}\cup\Omega_i^{(g)}$, where
\[
    \Omega_i^{(f)}\coloneqq\left\{x\in3Q_0\,\colon\,\mathcal{M}_pf_i(x)>\Big(\frac{100^dn}{1-\gamma}\Big)^\frac{1}{p}\| f_i\|_{\textit{\L}^p(Q_0,B_1)}\right\}
\] 
and
\[
    \Omega_i^{(g)}\coloneqq\left\{x\in3Q_0\,\colon\,\mathcal{M}_{q'}g_i(x)>\Big(\frac{100^dn}{1-\gamma}\Big)^\frac{1}{q'}\| g_i\|_{\textit{\L}^{q'}(3Q_0,B_2^*)}\right\}.
\]
We write $\Omega\coloneqq\bigcup_{i=1}^n\Omega_i$ and note that by the $L^1\to L^{1,\infty}$ boundedness of $\mathcal{M}$ we have
\begin{align*}|\Omega|\leq \sums{i}{1}{n}|\Omega_i^{(f)}|+|\Omega_i^{(g)}|&\leq \sums{i}{1}{n}5^d\frac{1-\gamma}{100^dn}(|Q_0|+3^d|Q_0|)\\&=\frac{3^d+1}{20^dn}(1-\gamma)|Q_0|<(1-\gamma)|Q_0|\end{align*}
and by setting $E_{Q_0}\coloneqq Q_0\setminus\Omega$ we have $|E_{Q_0}|\geq \gamma |Q_0|$.

We perform a Whitney decomposition $\mathcal{W}$ of $\Omega$ consisting of disjoint dyadic cubes $P$ with side length $2^{L(P)}\in2^\Z$, so that
\[
    5\operatorname{diam}(P)\leq \operatorname{dist}(P,\Omega^\complement)\leq 12\operatorname{diam}(P),\quad\text{for all } P\in\mathcal{W}.
\] 
We set for $i=1,\dots,n,$
\begin{align*}
    b_{i,P}=(f_i-\fint_P f_i)\mathbbm{1}_P,\qquad
    b_i=\sum_{\substack{P\in\mathcal{W}\\P\subset Q_0}}b_{i,P},
\end{align*}
and
\begin{align*}
    h_i\coloneqq f_i\mathbbm{1}_{\Omega^\complement}+\sum_{\substack{P\in\mathcal{W}\\P\subset Q_0}}\mathbbm{1}_P\fint_P f_i
\end{align*}
We were able to take cubes $P$ that satisfy $P\subset Q_0$, since $\operatorname{supp}(f_i)\subset Q_0$. Indeed, if $P\cap Q_0\neq \emptyset$, then $P$ is a true subset of $Q_0$. Otherwise, we would have $Q_0\subset P$ as $Q_0$ and $P$ are dyadic, but this leads to \[|Q_0|\leq |P|\leq |\Omega|<(1-\gamma)|Q_0|,\] which is a contradiction.
Now we have the Calderon-Zygmund decompositions $f_i=h_i+b_i$ and it follows that
\[\langle T\Vec{f},\vec g\,\rangle=\langle T\Vec{h},\vec g\,\rangle+\langle T\Vec{b},\vec g\,\rangle .\]
For a dyadic cube $Q$ with side length $2^{L(Q)}\in2^\Z$, we define \[S_Qf\coloneqq\sum_{N_1\leq j\leq L(Q)}T_j[f\mathbbm{1}_Q].\] 
Note that for small enough $P$ the function $f_i$ will be a constant on $P$, which implies that $b_{i,P}=0$.
Thus without issues of convergence we can write
\begin{align*}
    \langle T\Vec{b},\vec g\,\rangle&=\sums{i}{1}{n} \langle Tb_i,g_i\rangle=\sum_{P\in\mathcal{W}}\sums{i}{1}{n} \langle Tb_{i,P},g_i\rangle\\
    &= \sum_{P\in\mathcal{W}}\sums{i}{1}{n}\bigg( \langle S_Pf_i,g_i\rangle-\langle S_P\Big[\mathbbm{1}_P\fint_Pf_i\Big],g_i\rangle\\&\qquad\qquad\qquad\qquad\qquad\qquad+\langle (T-S_P)b_{i,P},g_i\rangle\bigg)
    \\&\eqqcolon I+II+III.
\end{align*}

We handle the first term $I$ with the lower scale quantity $U(N-1)$. Indeed, since each $P$ that contributes a non-zero summand  in the first term is a proper subcube of $Q_0$, we have $L(P)-N_1\leq N-1$ and hence by definition
\[|\langle S_P\Vec{f},\vec g\,\rangle|\leq U(N-1)\mathfrak{G}_P(\Vec{f}\mathbbm{1}_P,\Vec{g}).\]
Meaning that, given any $\varepsilon>0$, there exists a $\gamma$-sparse collection $\mathfrak{S}_{P,\varepsilon}\subset \mathscr{D}(P)$, such that
\[
    |\langle S_P\Vec{f},\vec g\,\rangle| \leq (U(N-1)+\varepsilon)\sum_{Q\in \mathfrak{S}_{P,\varepsilon}}|3Q|\llangle \Vec{f}\rrangle_{\textit{\L}^p(3Q,B_1)}\cdot\llangle\Vec{g}\rrangle_{\textit{\L}^{q'}(3Q,B_2^*)}.
\]
We combine these collections
\[\mathfrak{S}'_\varepsilon\coloneqq \bigcup_{\substack{P\in\mathcal{W}\\P\subset Q_0}}\mathfrak{S}_{P,\varepsilon},\]
and get
\begin{align}|I|&\leq (U(N-1)+\varepsilon)\sum_{Q\in \mathfrak{S}'_{\varepsilon}}|3Q|\llangle \Vec{f}\rrangle_{\textit{\L}^p(3Q,B_1)}\cdot\llangle\Vec{g}\rrangle_{\textit{\L}^{q'}(3Q,B_2^*)}.\numberthis\label{badOneEst}\end{align}

Note that
\begin{align*}
    |\langle T\Vec{h},\vec g\,\rangle+ II+III|=|\langle T\Vec{f},\vec g\,\rangle-I|&=|\langle T\Vec{f},\vec g\,\rangle- \sum_{P\in\mathcal{W}}\langle S_P\Vec{f},\vec g\,\rangle|.
\end{align*}

\textbf{Claim:} For all $i=1,\dots,n$ we have
\begin{align*}
    |\langle Tf_i, g_i\rangle- \sum_{P\in\mathcal{W}}&\langle S_Pf_i,g_i\rangle|\\&\lesssim_{d,p,q,\gamma,\kappa} \mathcal{C}n^{\frac{1}{p}+\frac{1}{q'}}\,|3Q_0|\,\|f_i\|_{\textit{\L}^p(3Q_0,B_1)}\|g_i\|_{\textit{\L}^{q'}(3Q_0,B_2^*)}\numberthis\label{TheClaim}.
\end{align*}
Taking the above claim for granted and applying Lemma \ref{TH1}, we obtain 
\begin{align*}
    |\langle T\Vec{f},\vec g\,\rangle-\sum_{P\in\mathcal{W}}\langle S_P\Vec{f},&\vec g\,\rangle|\leq \sums{i}{1}{n}|\langle T f_i, g_i\rangle-\sum_{P\in\mathcal{W}}\langle S_P f_i,g_i\rangle|\\&\lesssim_{d,p,q,\gamma,\kappa}\mathcal{C}n^{\frac{1}{p}+\frac{1}{q'}}\,|3Q_0|\,\sums{i}{1}{n}\| f_i\|_{\textit{\L}^p(3Q_0,B_1)}\|g_i\|_{\textit{\L}^{q'}(3Q_0,B_2^*)}
    \\&\leq \mathcal{C}n^{\frac{3}{2}+\frac{1}{p}+\frac{1}{q'}}\,|3Q_0|\,\llangle \Vec{f}\rrangle_{\textit{\L}^p(3Q_0,B_1)}\cdot\llangle \Vec{g}\rrangle_{\textit{\L}^{q'}(3Q_0,B_2^*)}.\numberthis\label{normToBody}
\end{align*}

Applying \eqref{badOneEst} and \eqref{normToBody} we see that with the collection \[\mathfrak{S}_{\varepsilon}\coloneqq\{Q_0\}\cup\mathfrak{S}'_{\varepsilon},\] which is a $\gamma$-sparse collection of cubes in $\mathscr{D}(Q_0)$, we may estimate
\begin{align*}
    |\langle T\vec f,\vec g\rangle|&\leq |I|+|\langle T\vec f,\vec g\rangle-I|\\
    &\leq \max\{U(N-1)+\varepsilon,\mathcal{C}n^{\frac{3}{2}+\frac{1}{p}+\frac{1}{q'}}\}\sum_{Q\in \mathfrak{S}_{\varepsilon}}|3Q|\llangle \Vec{f}\rrangle_{\textit{\L}^p(3Q,B_1)}\cdot\llangle\Vec{g}\rrangle_{\textit{\L}^{q'}(3Q,B_2^*)}\\
    &\leq\max\{U(N-1)+\varepsilon,\mathcal{C}n^{\frac{3}{2}+\frac{1}{p}+\frac{1}{q'}}\}\,\mathfrak{G}_{Q_0}(\Vec{f},\Vec{g})
\end{align*}
and letting $\varepsilon\to0$ concludes the proof.
\end{proof}
We still need to prove inequality \eqref{TheClaim}, i.e., the claim that we made during the proof of the above lemma. The claim is almost identical to a step in the proof of the main result in \cite{beltran_multi-scale_2020}. The difference being that we have to control $n$ possibly distinct functions with a single set $\Omega$. The proof is quite long and  dependence on $n$ is arguably the only novelty in our situation. On the other hand the dependence on $n$ does play a role and it will appear throughout the proof. Thus for the readers convenience we will include a proof in appendix A, see Proposition \ref{appendixProp}. 

As we previously mentioned, now that we have proven Lemma \ref{inductiveLemma}, the rest of the proof of Theorem \ref{mainresult} is similar to Section 4.2 of \cite{beltran_multi-scale_2020}.

\section{Matrix \texorpdfstring{$A_p$}{Ap} weights and norm inequalities}
A matrix weight is a locally integrable function $W \colon \R^d \to \R^{
n\times n}$ that is almost everywhere
positive definite -valued. For a Banach space $B$ the space
$L^r_{B^n}(W)$ consists of all measurable $\vec f \colon \R^d \to B^n$ such that $W^\frac{1}{r} \vec f \in L^r_{B^n}(\R^d)$,
and $\|
\vec f\|_{L^r_{B^n}(W)}
\coloneqq \|W^\frac{1}{r} \vec f\|_{L^r_{B^n}(\R^d)}$. For a matrix weight $W \colon \R^d \to \R^{
n\times n}$ and $1<r<\infty$, we
use the definition introduced in \cite{roudenko_matrix-weighted_2002}, i.e.,
\[
    [W]_{A_r}\coloneqq\sup\limits_{Q}\fint_Q\left(\fint_Q\Big|W^{\frac{1}{r}}(x)W^{-\frac{1}{r}}(y)\Big|^{r'}_{op}\intD y\right)^\frac{r}{r'}\intD x
\]
and say that $W\in A_r$ if $[W]_{A_r}<\infty$. It is known that with $\Sigma=W^{-\frac{r'}{r}}$ we have
\[
    [W]_{A_r}^\frac{1}{r}\eqsim_{n,r}[\Sigma]_{A_{r'}}^\frac{1}{r'}.
\]
See \cite[proof of Lemma 3.2]{domelevo_boundedness_2021} for more details on the above fact.
For $y\in\R^n$ we also define the associated scalar weights 
\[
    W^{Sc}_{r,y}(x)\coloneqq \left|W^\frac{1}{r}(x)y\right|^r.
\]
In \cite[Corollary 2.2]{goldberg_matrix_2003} it was shown that for any $y\in \R^n$ the scalar weight $W^{Sc}_{r,y}$ is in $A_r$ and the weight constant satisfies \[[W]^{Sc}_{A_r}\coloneqq\sup_{y\in\R^n}[W^{Sc}_{r,y}]_{A_r}\lesssim_{n,r} [W]_{A_r}.\] 

As an application of the John ellipsoid theorem, one can reduce a norm $\rho$ in $\R^n$ to an ellipsoid. More explicitly, there exists a self-adjoint matrix $A$ such that $\rho(x)\eqsim_n |Ax|$ for every $x\in \R^n$. Such a matrix $A$ is called the reducing matrix (or operator) of $\rho$. We will make great use of the reducing matrices $A_{Q,r}$ and $B_{Q,r}$ that give 
\[
    |A_{Q,r}x|\eqsim_{n}\left(\fint_Q\left|W^\frac{1}{r}(y)x\right|^r\intD y\right)^\frac{1}{r}
\]
and
\[
    |B_{Q,r}x|\eqsim_{n}\left(\fint_Q\left|W^{-\frac{1}{r}}(y)x\right|^{r'}\intD y\right)^\frac{1}{r'},
\]
for all $x\in\R^n$. In \cite[proof of Lemma 1.3]{roudenko_matrix-weighted_2002} it is shown that 
\[
    |A_{Q,r}B_{Q,r}|_{op}\eqsim_{n,r}[W]_{A_r}^\frac{1}{r}.
\]
The following lemma is essentially Proposition 2.4 from \cite{goldberg_matrix_2003}. 

\begin{lem}\label{redOpRevHölderLemma}
    Assume that $W\in A_r$ and let $\Sigma=W^{-\frac{r'}{r}}$. Let $A_Q\coloneqq A_{Q,r}$ 
    and $B_Q\coloneqq B_{Q,r}$.
    Then with $\delta = \frac{1}{2^{d+1}[W]^{Sc}_{A_r}-1}$ and $\sigma=\frac{1}{2^{d+1}[\Sigma]^{Sc}_{A_{r'}}-1}$ we have
    \[
        \sup_Q \left(\fint_Q\left|A_Q^{-1}W^\frac{1}{r}(x)\right|_{op}^{r(1+\delta)}\intD x\right)^\frac{1}{r(1+\delta)} \lesssim_{n} 1
    \]
    and 
    \[
        \sup_Q \left(\fint_Q\left|B_Q^{-1}W^{-\frac{1}{r}}(x)\right|_{op}^{r'(1+\sigma)}\intD x\right)^\frac{1}{r'(1+\sigma)} \lesssim_{n}1.
    \]
\end{lem}
\noindent The following classical inequality on operator norm of powers of matrices will turn out to be helpful.
\begin{lem}[Cordes inequality, \cite{cordes_spectral_1987}]\label{Cordesineq}
    Let $A$ and $B$ be self-adjoint positive definite matrices and $0<\alpha<1$. Then 
    \[
        \left|A^\alpha B^\alpha\right|_{op}\leq |AB|_{op}^\alpha.
    \]
\end{lem}
 \noindent We will also need the following reverse Hölder class.
\begin{dfntn}\label{RHtsclass}
    A matrix weight $W$ is a $RH_{t,s}$ weight, if \[[W]_{RH_{t,s}}\coloneqq\sup_{y\in\R^n}\sup_{Q}\left(\fint_Q\left|W^\frac{1}{t}(x)y\right|^{ts}\intD x\right)^\frac{1}{s}\left(\fint_Q\left|W^\frac{1}{t}(x)y\right|^{t}\intD x\right)^{-1}<\infty.\] 
\end{dfntn} 
\noindent Another way of formulating the above definition is to say that $W\in RH_{t,s}$ if $W^{Sc}_{t,y}$ is in the scalar reverse Hölder class $RH_s$ uniformly in $y$. This motivates the following lemma.

\begin{lem}\label{scarevHölrmrk}
    Let $w$ be a scalar weight and $s>1$. Then 
    \begin{align}\label{RHequiAinftyone}
    w\in RH_s \quad\Leftrightarrow \quad w^s\in A_\infty.
    \end{align} Furthermore,
    \begin{align}\label{RHequiAinftytwo}
        \frac{[w^s]_{A_\infty}^\frac{1}{s}}{[w]_{A_\infty}}\leq [w]_{RH_s}\leq [w^s]_{A_\infty}^\frac{1}{s}.
    \end{align}
\end{lem}
\noindent See \cite[Corollary 6.2]{stromberg_fractional_1985} for the proof of \eqref{RHequiAinftyone} and \cite[Lemma 3.1]{johnson_homeomorphisms_1987} for the proof of \eqref{RHequiAinftytwo}. 

We also highlight an elementary inequality for positive numbers, which will make some computations more presentable.
\begin{lem}\label{numbersLemma}
    Let $\alpha,\beta\in \left[1,\infty\right[$, $\eta\in\left]0,1\right[$ and $\lambda=\alpha(\beta'(1+\eta))'$. Then 
    \[
        \left(\frac{1}{\eta}\right)^\frac{1}{\lambda}\lesssim \left(\frac{1}{\eta}\right)^\frac{1}{\alpha\beta}.
    \]
\end{lem}
\begin{proof}
    We simply write
    \[
        \lambda=\frac{\alpha\beta(1+\eta)}{1+\beta\eta}
    \]
    and thus
    \begin{align*}
        \left(\frac{1}{\eta}\right)^\frac{1}{\lambda}= \left(\left[\frac{1}{\eta}\right]^\frac{1}{\alpha\beta}\right)^{\frac{1+\beta\eta}{1+\eta}}&=\left(\left[\frac{1}{\eta}\right]^\frac{1}{\alpha\beta}\right)^{\frac{1}{1+\eta}}\left(\left[\frac{1}{\eta}\right]^\frac{1}{\alpha}\right)^{\frac{\eta}{1+\eta}}\\&\leq \left(\frac{1}{\eta}\right)^\frac{1}{\alpha\beta}\left(\frac{1}{\eta}\right)^{\eta}\\&\leq e^\frac{1}{e}\left(\frac{1}{\eta}\right)^\frac{1}{\alpha\beta}.
    \end{align*}
\end{proof}

The goal now is to make use of convex body domination to get a matrix weighted norm inequality. For the sake of maximal applicability we simply assume that $T$ satisfies the convex body domination inequality
\[
    |\langle T\vec f,\vec g\rangle|\leq C_1\sum_{Q\in\mathscr{Q}} |Q| \,\llangle\Vec{f}\rrangle_{\textit{\L}^p(Q,B_1)}\cdot \llangle\Vec{g}\rrangle_{\textit{\L}^{q'}(Q,B_2^*)},
\]
where $C_1$ is a constant independent of $\vec f$ and $\vec g$ and $\mathscr Q$ is a $\gamma$-sparse collection of cubes possibly depending on $\vec f$ and $\vec g$. Then
we want to find a new constant $C=C(C_1,W)$ such that
\[
    \|T\vec f\|_{L_{B_2^n}^r(W)}\leq C \|\vec f\|_{L_{B_1^n}^r(W)}.
\]
Note that for a Banach space $B$ and a measure space $S$ the Bochner space $L^{p'}_{B^*}(S)\subset(L^{p}_B(S))^*$ is norming for $L^{p}_B(S)$ (see \cite{hytonen_analysis_2016} Proposition 1.3.1). Thus the above inequality is equivalent to
\[
    \langle W^\frac{1}{r} T[W^{-\frac{1}{r}}\Vec{f}\,],\Vec{g}\,\rangle
    \leq C \|\vec f\|_{L^r_{B_1^n}}\|\vec g\|_{L^{r'}_{(B^*_2)^n}}.
\]
Directly from the assumed convex body domination, we get
\begin{align*}
    \langle W^\frac{1}{r}T[W^{-\frac{1}{r}}\vec f\,],\vec g\,\rangle&= \langle T[W^{-\frac{1}{r}}\vec f\,],W^\frac{1}{r}\vec g\,\rangle
    \\&\leq C_1\sum_{Q\in\mathscr Q}|Q|\llangle W^{-\frac{1}{r}}\Vec{f}\rrangle_{\textit{\L}^p(Q,B_1)}\cdot\llangle W^{\frac{1}{r}}\Vec{g}\rrangle_{\textit{\L}^{q'}(Q,B_2^*)}.
\end{align*}
We estimate the convex bodies in the following lemma.
\begin{lem}\label{convbodytointegralstuff}
    Let $W$ be a matrix weight and $p,r,q'\in\left[1, \infty\right[$. For $\vec f\in\textit{\L}^p_{B_1^n}$ and $\vec g\in \textit{\L}^{q'}_{(B_2^*)^n}$ we have
    \begin{align*}
        \llangle W^{-\frac{1}{r}}\Vec{f}\rrangle_{\textit{\L}^p(Q,B_1)}&\cdot\llangle W^{\frac{1}{r}}\Vec{g}\rrangle_{\textit{\L}^{q'}(Q,B_2^*)}\\\leq& \bigg(\fint_Q  |\vec g(y)|^{q'}_{(B_2^*)^n}\Big[\fint_Q\left|W^{-\frac{1}{r}}(x) W^\frac{1}{r}(y)\right|_{op}^p|\vec f(x)|^p_{B_1^n}\intD x\Big]^\frac{q'}{p}\intD y\bigg)^\frac{1}{q'}.
    \end{align*}
\end{lem}
\begin{proof}
 In the light of Lemma \ref{convBodyCalc}, for every $\varepsilon>0$ it holds that
\begin{align*}
    (1+\varepsilon&)^{-2}\,\llangle W^{-\frac{1}{r}}\Vec{f}\rrangle_{\textit{\L}^p(Q,B_1)}\cdot\llangle W^{\frac{1}{r}}\Vec{g}\rrangle_{\textit{\L}^{q'}(Q,B_2^*)}\leq \\&\quad\Big|\fint_QW^{-\frac{1}{r}}(x)\langle\vec f(x),\phi(x)\rangle_{(B_1,B_1^*)}\intD x\cdot\fint_Q W^\frac{1}{r}(y)\langle\vec g(y),\psi(y)\rangle_{(B_2^*,B_2^{**})}\intD y\Big|,
\end{align*}
for some $\phi\in \Bar{B}_{\textit{\L}^{p'}_{B_1^*}(Q)}$ and $\psi\in\Bar{B}_{ \textit{\L}^{q}_{B_2^{**}(Q)}}$. Furthermore,
\begin{align*}
    &\left|\fint_Q W^{-\frac{1}{r}}(x)\langle\vec f(x),\phi(x)\rangle_{(B_1,B_1^*)}\intD x\cdot\fint_Q W^\frac{1}{r}(y)\langle\vec g(y),\psi(y)\rangle_{(B_2^*,B_2^{**})}\intD y\right|\\& \quad= \left|\fint_Q \fint_QW^\frac{1}{r}(y)W^{-\frac{1}{r}}(x) \langle\vec f(x),\phi(x)\rangle_{(B_1,B_1^*)}\cdot\langle\vec g(y),\psi(y)\rangle_{(B_2^*,B_2^{**})}\intD x\intD y\right|
    \\& \quad\leq\fint_Q \fint_Q\left|W^{-\frac{1}{r}}(x) W^\frac{1}{r}(y)\right|_{op}|\vec f(x)|_{B_1^n}|\phi(x)|_{B_1^*}|\vec g(y)|_{(B_2^*)^n}|\psi(y)|_{B_2^{**}}\intD x\intD y.
\end{align*}
Applying Hölder's inequality twice we see that the last quantity is dominated by 
\begin{align*}
     \bigg(\fint_Q  |\vec g(y)|^{q'}_{(B_2^*)^n}\Big[\fint_Q\left|W^{-\frac{1}{r}}(x) W^\frac{1}{r}(y)\right|_{op}^p|\vec f(x)|^p_{B_1^n}\intD x\Big]^\frac{q'}{p}\intD y\bigg)^\frac{1}{q'}.
\end{align*}
Combining the estimates and letting $\varepsilon\to 0$ concludes the proof.
\end{proof}
\noindent As a warm up we cover the endpoint case first.

\textbf{Case $\bm{p=1}$, $\bm{q=\infty}$:} We choose $p=1$, $q=\infty$ and $r=2$. In this case Lemma \ref{convbodytointegralstuff} yields
\begin{align*}
    &\sum_{Q\in\mathscr Q}|Q|\llangle W^{-\frac{1}{2}}\Vec{f}\rrangle_{\textit{\L}^{1}(Q,B_1)}\cdot\llangle W^{\frac{1}{2}}\Vec{g}\rrangle_{\textit{\L}^{1}(Q,B_2^*)}\leq \int_{\R^d} \Tilde{L}(|\vec f|_{B^n_1})(y)|\vec g(y)|_{(B_2^*)^n}\intD y,
\end{align*}
where 
\[
    \Tilde{L}f(y)\coloneqq\sum_{Q\in\mathscr Q}\mathbbm 1_Q(y)\fint_Q \left|W^{-\frac{1}{2}}(x)W^\frac{1}{2}(y)\right|_{op}|f(x)|\intD x
\]
for a scalar valued function $f$ and $\Tilde{L}(|\vec f|_{B^n_1})(y)\coloneqq\Tilde{L}(x\mapsto|\vec f(x)|_{B^n_1})(y)$.
In \cite[Lemma 5.6]{nazarov_convex_2017} it was shown that $\Tilde{L}$ is bounded in $L^2$ with 
\[
    \|\Tilde{L}\|_{L^2\to L^2}\lesssim_{d,n,\gamma} [W]_{A_2}^\frac{3}{2}.
\]
Thus we get a matrix weighted $A_2$ norm inequality
\begin{align*}
    \|T\vec f\|_{L^2_{B_2^n}(W)}\lesssim_{d,n,\gamma}C_1\,[W]_{A_2}^\frac{3}{2}\|\vec f\|_{L^2_{B_1^n}(W)}.
\end{align*}
The extrapolation result of \cite{bownik_extrapolation_2022} now gives that for all $1<r<\infty$, we have
\[
    \|T\vec f\|_{L^r_{B_2^n}(W)}\lesssim_{d,n,r,\gamma}C_1\,[W]_{A_r}^{\frac{3}{2}\max\{1,\frac{r'}{r}\}}\|\vec f\|_{L^r_{B_1^n}(W)}.\numberthis\label{extrapomweightbound}
\]
If we want to study the whole range $1<r<\infty$ without extrapolation arguments, then the weighted norm inequality is reduced to studying the $L^r$-boundedness of the operator $\Tilde{L}_r$ defined by
\[
    \Tilde{L}_rf(y)\coloneqq\sum_{Q\in\mathscr Q}\mathbbm 1_Q(y)\fint_Q \left|W^{-\frac{1}{r}}(x)W^\frac{1}{r}(y)\right|_{op}|f(x)|\intD x.
\]
We study the boundedness of the above operator in the following proposition.
\begin{prop}\label{OneoneCase}
    Let $1<r<\infty$. Then the operator $\Tilde{L}_r$ is bounded on $L^r$ with
    \[
        \|\Tilde{L}_r\|_{L^r\to L^r}\lesssim_{d,n,r,\gamma} [W]_{A_r}^{1+\frac{1}{r-1}-\frac{1}{r}}
    \]
\end{prop}
\begin{rmrk}
    The exponent $1+\frac{1}{r-1}-\frac{1}{r}$ can also be found in \cite[Corollary 1.16]{cruz-uribe_ofs_two_2018} and \cite[Theorem 1.5]{duong_variation_2021}. Our proof has the same idea, but it is more elementary in the sense that we do not use any Orlicz space theory. 
\end{rmrk}
\begin{proof}
    Let $A_Q$ and $B_Q$ be as in Lemma \ref{redOpRevHölderLemma}.
    Since $|A_QB_Q|_{op}\eqsim_{n,r}[W]_{A_r}^\frac{1}{r}$ we get
    \begin{align*}
        \left|W^{-\frac{1}{r}}(x)W^\frac{1}{r}(y)\right|_{op}\lesssim_{n,r}[W]_{A_r}^\frac{1}{r}\left|A_Q^{-1}W^\frac{1}{r}(y)\right|_{op}\left|B_Q^{-1}W^{-\frac{1}{r}}(x)\right|_{op}
    \end{align*}
    and hence 
    \begin{align*}
        \langle L_rf,g\rangle_{(L^r,L^{r'})}&\leq \sum_{Q\in\mathscr Q}|Q|\fint_Q\fint_Q \left|W^{-\frac{1}{r}}(x)W^\frac{1}{r}(y)\right|_{op}|f(x)|\intD x \,|g(y)|\intD y \\
        &\lesssim_{n,r} [W]_{A_r}^\frac{1}{r}\sum_{Q\in\mathscr Q}|Q|\fint_Q \left|B_Q^{-1}W^{-\frac{1}{r}}(x)\right|_{op}|f(x)|\intD x \\&\qquad\qquad\qquad\qquad\qquad\qquad\fint_Q\left|A_Q^{-1}W^{\frac{1}{r}}(y)\right|_{op}|g(y)|\intD y.
    \end{align*}
    By Hölder's inequality we have
    \[
    \sum_{Q\in\mathscr Q}|Q|\fint_Q \left|B_Q^{-1}W^{-\frac{1}{r}}(x)\right|_{op}|f(x)|\intD x \,\fint_Q\left|A_Q^{-1}W^{\frac{1}{r}}(y)\right|_{op}|g(y)|\intD y \leq 
    S_{f}\,S_{g},
    \]
    where we denoted
    \[
        S_{f}\coloneqq\left(\sum_{Q\in\mathscr Q}|Q|\left[\fint_Q \left|B^{-1}_QW^{-\frac{1}{r}}(x)\right|_{op}|f(x)|\intD x \right]^r\right)^\frac{1}{r}
    \]
    and 
    \[        S_{g}\coloneqq\left(\sum_{Q\in\mathscr Q}|Q|\left[\fint_Q \left|A_Q^{-1}W^\frac{1}{r}(x)\right|_{op}|g(x)|\intD x \right]^{r'}\right)^\frac{1}{r'}
    \]
    
    To estimate $S_f$, we use Hölder's inequality to get
    \begin{align*}
        \fint_Q \left|B_Q^{-1}W^{-\frac{1}{r}}(x)\right|_{op}&|f(x)|\intD x \\\leq& \left(\fint_Q \left|B_Q^{-1}W^{-\frac{1}{r}}(x)\right|_{op}^{r'(1+\sigma)}\intD x\right)^\frac{1}{r'(1+\sigma)}\left(\fint_Q |f(x)|^{a}\intD x\right)^\frac{1}{a},
    \end{align*}
    where $\sigma\coloneqq\frac{1}{2^{d+1}[\Sigma]^{Sc}_{A_r'}-1}$ and $a\coloneqq(r'(1+\sigma))'<r$.
     By Lemma \ref{redOpRevHölderLemma} we have that
    \[
        \left(\fint_Q \left|B_Q^{-1}W^{\frac{1}{r}}(x)\right|_{op}^{r'(1+\sigma)}\intD x\right)^\frac{1}{r'(1+\sigma)}\lesssim_{n}1
    \]
    Thus applying the sparseness of the cubes we get
    \begin{align*}
        S_{f}\lesssim_{n} \left(\sum_{Q\in\mathscr Q}|Q|\left[\fint_Q |f(x)|^a\intD x \right]^\frac{r}{a}\right)^\frac{1}{r}&\leq \left(\sum_{Q\in\mathscr Q}\gamma^{-1}\int_{E_Q}\left[\mathcal{M}_a^\mathscr D f(y)\right]^r\intD y\right)^\frac{1}{r}\\&
        \leq\gamma^{-\frac{1}{r}}\|\mathcal{M}^\mathscr D_af\|_{L^r},
    \end{align*}
    where $\mathcal{M}^\mathscr D_af=(\mathcal{M}^\mathscr D|f|^a)^\frac{1}{a}$ with $\mathcal{M}^\mathscr D$ being the dyadic Hardy-Littlewood maximal operator.
    Similarly, with $\delta \coloneqq \frac{1}{2^{d+1}[W]^{Sc}_{A_r}-1}$ we have
    \[
        S_{g}\lesssim_{n,\gamma} \|\mathcal{M}^\mathscr D_bg\|_{L^{r'}},
    \]
    where $b\coloneqq(r(1+\delta))'<r'$.

    Combining everything together we get
    \begin{align*}
        \langle L_rf,g\rangle_{(L^r,L^{r'})}&\lesssim_{n,r,\gamma} [W]_{A_r}^\frac{1}{r}\|\mathcal{M}^\mathscr D_af\|_{L^{r}}\|\mathcal{M}^\mathscr D_bg\|_{L^{r'}} \\&\leq [W]_{A_r}^\frac{1}{r}\left(\frac{r}{r-a}\right)^\frac{1}{a}\left(\frac{r'}{r'-b}\right)^\frac{1}{b}\|f\|_{L^{r}}\|g\|_{L^{r'}}.
    \end{align*}
    For $y$,  and $x=y(1+\delta)$ we have 
    \[
        y'-x' =\frac{y\delta}{(y-1)(y\delta+y-1)}.
    \]
    Applying this with $y=r$ yields that $r'-b\gtrsim_r\delta$ and similarly $r-a\gtrsim_r\sigma$. We may then apply Lemma \ref{numbersLemma} twice, first with $\alpha=1$, $\beta=r$ and $\eta=\sigma$ and then with $\alpha=1$ $\beta=r'$ and $\eta=\delta$, to get
    \[
        \left(\frac{r}{r-a}\right)^\frac{1}{a}\left(\frac{r'}{r'-b}\right)^\frac{1}{b}\lesssim_r \left(\frac{1}{\sigma}\right)^{\frac{1}{a}}\left(\frac{1}{\delta}\right)^{\frac{1}{b}}\lesssim \left(\frac{1}{\sigma}\right)^{\frac{1}{r}}\left(\frac{1}{\delta}\right)^{\frac{1}{r'}}.
    \]
    Then we are done once we note that
    \[
        \left(\frac{1}{\sigma}\right)^{\frac{1}{r}}\left(\frac{1}{\delta}\right)^{\frac{1}{r'}}\lesssim_{d,n,r}[\Sigma]_{A_{r'}}^\frac{1}{r}[W]_{A_r}^\frac{1}{r'}\eqsim_{n,r}[W]_{A_r}^{1+\frac{1}{r-1}-\frac{2}{r}}.
    \]
\end{proof}
\noindent Now it follows that if $1<r<\infty$ and $W\in A_r$, then we have
\[
    \|T\vec f\|_{L^r_{B_2^n}(W)}\lesssim_{d,n,r,\gamma}C_1\,[W]_{A_r}^{1+\frac{1}{r-1}-\frac{1}{r}}\|\vec f\|_{L^r_{B_1^n}(W)}.\numberthis\label{poneqinftynormest}
\]
We note that the exponent of $[W]_{A_r}$ in the above estimate is better than the one in \eqref{extrapomweightbound} that arose from the extrapolation argument. As we have already discussed, similar results have been achieved in \cite{cruz-uribe_ofs_two_2018, duong_variation_2021}. The novelty in the above result comes from the fact that we consider Banach space -valued functions and we are dominating a bilinear form. In the scalar case it is known, see \cite[proof of Theorem 1]{perez_sharp_2010}, that one has the sharp weight exponent $\max\{1,\frac{1}{r-1}\}$.

\textbf{Case $\bm{p>1}$, $\bm{q<\infty}$:}
Recall that in this case Lemma \ref{convbodytointegralstuff} gives 
\begin{align*}
    &\langle W^\frac{1}{r}T[W^{-\frac{1}{r}}\vec f\,],\vec g\,\rangle\\&\qquad\leq  C_1\sum_{Q\in\mathscr Q}|Q|\bigg(\fint_Q  |\vec g(y)|^{q'}_{(B_2^*)^n}\Big[\fint_Q\left|W^{-\frac{1}{r}}(x) W^\frac{1}{r}(y)\right|_{op}^p|\vec f(x)|^p_{B_1^n}\intD x\Big]^\frac{q'}{p}\intD y\bigg)^\frac{1}{q'}.
\end{align*}
Let $t\coloneqq\frac{r}{p}$. We assume that $p<r$ and $W\in A_t$. 
The term inside the sum can be dominated  with  
\begin{align*}
    |Q|\left|A_{Q,t}^\frac{1}{p}B_{Q,t}^\frac{1}{p}\right|_{op}&\left(\fint_Q |\vec f(x)|^{p}_{B_1^n}\left|B^{-\frac{1}{p}}_{Q,t}W^{-\frac{1}{r}}(x)\right|_{op}^{p}\intD x \right)^\frac{1}{p}\\&\qquad\qquad\qquad\qquad\left(\fint_Q |\vec g(y)|^{q'}_{(B_2^*)^n}\left|A^{-\frac{1}{p}}_{Q,t}W^\frac{1}{r}(y)\right|_{op}^{q'}\intD y \right)^\frac{1}{q'}.
\end{align*}
By the Cordes inequality (Lemma \ref{Cordesineq}) we have
\[
    \left|A_{Q,t}^\frac{1}{p}B_{Q,t}^\frac{1}{p}\right|_{op}\leq \left|A_{Q,t}B_{Q,t}\right|_{op}^\frac{1}{p}\eqsim_{n,p,r} [W]_{A_t}^\frac{1}{r},
\]
\[
    \left|B^{-\frac{1}{p}}_{Q,t}W^{-\frac{1}{r}}(x)\right|_{op}^{p}\leq \left|B^{-1}_{Q,t}W^{-\frac{1}{t}}(x)\right|_{op}
\]
and
\[
    \left|A^{-\frac{1}{p}}_{Q,t}W^\frac{1}{r}(y)\right|_{op}^{q'}\leq \left|A^{-1}_{Q,t}W^\frac{1}{t}(y)\right|_{op}^\frac{q'}{p}.
\]
We combine the estimates so far and use Hölder's inequality to get
\begin{align*}
    \langle W^\frac{1}{r}T[W^{-\frac{1}{r}}\vec f],\vec g\,\rangle\lesssim_{n,p,r} [W]_{A_t}^\frac{1}{r} S_{\vec f}\,\tilde S_{\vec g},
\end{align*}
where
\[
    S_{\vec f} \coloneqq \left(\sum_{Q\in\mathscr Q}|Q|\left[\fint_Q |\vec f(x)|^{p}_{B_1^n}\left|B^{-1}_{Q,t}W^{-\frac{1}{t}}(x)\right|_{op}\intD x \right]^\frac{r}{p}\right)^\frac{1}{r}
\]
and
\[
    \tilde S_{\vec g} \coloneqq \left(\sum_{Q\in\mathscr Q}|Q|\left[\fint_Q |\vec g(y)|^{q'}_{(B_2^*)^n}\left|A^{-1}_{Q,t}W^{\frac{1}{t}}(y)\right|_{op}^\frac{q'}{p}\intD y \right]^\frac{r'}{q'}\right)^\frac{1}{r'}.
\]
We will first deal with $S_{\vec f}$. An application of Hölder's inequality yields
\begin{align*}
    &\left(\fint_Q |\vec f(x)|^{p}_{B_1^n}\left|B^{-1}_{Q,t}W^{-\frac{1}{t}}(x)\right|_{op}\intD x \right)^\frac{1}{p}\\&\qquad\qquad\leq \left(\fint_Q |\vec f(x)|^{a}_{B_1^n}\right)^\frac{1}{a}\left(\fint_Q\left|B^{-1}_{Q,t}W^{-\frac{1}{t}}(x)\right|_{op}^{t'(1+\sigma)}\intD x \right)^\frac{1}{pt'(1+\sigma)},
\end{align*}
where $a\coloneqq p(t'(1+\sigma))'$ and $\sigma\coloneqq\frac{1}{2^{d+1}[\Sigma]^{Sc}_{A_{t'}}-1}$. Lemma \ref{redOpRevHölderLemma} gives that
\[
    \left(\fint_Q\left|B^{-1}_{Q,t}W^{-\frac{1}{t}}(x)\right|_{op}^{t'(1+\sigma)}\intD x \right)^\frac{1}{t'(1+\sigma)}\lesssim_{n} 1.
\]
Thus 
\begin{align*}
    S_{\vec f}&=\left(\sum_{Q\in\mathscr Q}|Q|\left[\fint_Q |\vec f(x)|^{p}_{B_1^n}\left|B^{-1}_{Q,t}W^{-\frac{1}{t}}(x)\right|_{op}\intD x \right]^\frac{r}{p}\right)^\frac{1}{r}\\&\lesssim_{n}\left(\sum_{Q\in\mathscr Q}|Q|\left[\fint_Q |\vec f(x)|^{a}_{B_1^n}\intD x \right]^\frac{r}{a}\right)^\frac{1}{r}.
\end{align*}
Note that $a=p(t'(1+\sigma))'<pt=r$. Thus similar calculations as in Proposition \ref{OneoneCase} yield
\begin{align*}
    \left(\sum_{Q\in\mathscr Q}|Q|\left[\fint_Q |\vec f(x)|^{a}_{B_1^n}\intD x \right]^\frac{r}{a}\right)^\frac{1}{r}\lesssim_{\gamma} \left(\frac{r}{r-a}\right)^\frac{1}{a}\| \vec f\|_{L^r_{B^n_1}}
\end{align*}
and hence
\begin{align*}
    S_{\vec f}&\lesssim_{n,\gamma} \left(\frac{r}{r-a}\right)^\frac{1}{a}\| \vec f\|_{L^r_{B^n_1}}.
\end{align*}
Furthermore, we have $r-a\gtrsim_{r,p}\sigma$ and hence
\[
    \left(\frac{r}{r-a}\right)^\frac{1}{a}\lesssim_{p,r} \sigma^{-\frac{1}{a}}\lesssim\sigma^{-\frac{1}{r}}\lesssim_{d,n,p,r}[W]_{A_t}^\frac{1}{r(t-1)},
\]
where the second inequality is due to Lemma \ref{numbersLemma} with $\alpha=p$, $\beta=t$ and $\eta=\sigma$.

In order to estimate $\tilde S_{\vec g}$
we assume $q>r$, fix $s\coloneqq\left(\frac{
q}{r}\right)'$ and consider weights $W\in RH_{t,s}$.
We then note that \[b\coloneqq q'\left(\frac{pts(1+\vartheta)}{q'}\right)' =
q'\left(\frac{rs(1+\vartheta)}{q'}\right)'<q'\left(\frac{rs}{q'}\right)'=r'\] and apply Hölder's inequality to get
\begin{align*}
    \left(\fint_Q |\vec g(y)|^{q'}_{(B_2^*)^n}\left|A^{-1}_{Q,t}W^\frac{1}{t}(y)\right|_{op}^\frac{q'}{p}\intD y \right)^\frac{1}{q'}\leq \left(\fint_Q \left|A^{-1}_{Q,t}W^\frac{1}{t}(y)\right|_{op}^{ts(1+\vartheta)}\intD y \right)^{\frac{1}{ts(1+\vartheta)}\frac{1}{p}}&\\\left(\fint_Q |\vec g(y)|^{b}_{(B_2^*)^n}\intD y\right)^\frac{1}{b}&,
\end{align*}
where \[\vartheta\coloneqq\frac{1}{2^{d+1}[(AW^{Sc}_{Q,t})^s]_{A_\infty}-1}\quad \text{ with } \quad AW^{Sc}_{Q,t}(y)\coloneqq\left|A^{-1}_{Q,t}W^\frac{1}{t}(y)\right|_{op}^{t}.\] The above is well defined since by Lemma \ref{scarevHölrmrk} the scalar weight $(AW^{Sc}_{Q,t})^s$ is in $A_\infty$, and we can apply the reverse Hölder inequality and the definition of $RH_{t,s}$ to get 
\begin{align*}
    \left(\fint_Q \left|A^{-1}_{Q,t}W^\frac{1}{t}(y)\right|_{op}^{ts(1+\vartheta)}\intD y \right)^\frac{1}{ts(1+\vartheta)}&\lesssim \left(\fint_Q \left|A^{-1}_{Q,t}W^\frac{1}{t}(y)\right|_{op}^{ts}\intD y \right)^\frac{1}{ts}\\&\leq [W]^\frac{1}{t}_{RH_{t,s}}\left(\fint_Q \left|A^{-1}_{Q,t}W^\frac{1}{t}(y)\right|_{op}^{t}\intD y \right)^\frac{1}{t}\\&\lesssim_{n}[W]^\frac{1}{t}_{RH_{t,s}}.
\end{align*}
We arrive at 
\begin{align*}
    \left(\fint_Q |\vec g(y)|^{q'}_{(B_2^*)^n}\left|A^{-\frac{1}{p}}_{Q,t}W^\frac{1}{r}(y)\right|_{op}^{q'}\intD y \right)^\frac{1}{q'}&\lesssim_{n} [W]^\frac{1}{r}_{RH_{t,s}} \left(\fint_Q |\vec g(y)|^{b}_{(B_2^*)^n}\right)^\frac{1}{b}
\end{align*}
and hence 
\begin{align*}
    \tilde S_{\vec g}&=\left(\sum_{Q\in\mathscr Q}|Q|\left[\fint_Q |\vec g(y)|^{q'}_{(B_2^*)^n}\left|A^{-\frac{1}{p}}_{Q,t}W^{\frac{1}{r}}(x)\right|_{op}^{q'}\intD x \right]^\frac{r'}{q'}\right)^\frac{1}{r'}\\&\lesssim_{n,\gamma} [W]^\frac{1}{r}_{RH_{t,s}}\left(\frac{r'}{r'-b}\right)^\frac{1}{b}\| \vec g\|_{L^{r'}_{(B_2^*)^n}}
    \lesssim_{q,r}[W]^\frac{1}{r}_{RH_{t,s}}\vartheta^{-\frac{1}{b}}\| \vec g\|_{L^{r'}_{(B_2^*)^n}}.
\end{align*}
We again have 
\[
    \vartheta^{-\frac{1}{b}}\lesssim \vartheta^{-\frac{1}{r'}},
\]
which can be seen by recalling that 
\[
    b=q'\left(\frac{rs}{q'}(1+\vartheta)\right)'
\]
and applying Lemma \ref{numbersLemma} with $\alpha=q'$, $\beta=\left(\frac{rs}{q'}\right)'=\frac{r'}{q'}$ and $\eta=\vartheta$.
By Lemma \ref{scarevHölrmrk} we have
\[
    [(AW^{Sc}_{Q,t})^s]_{A_\infty} \leq [AW^{Sc}_{Q,t}]_{RH_{s}}^s[AW^{Sc}_{Q,t}]_{A_\infty}^s,
\]
and the terms on the right-hand side satisfy
\[
    [AW^{Sc}_{Q,t}]_{RH_{s}}\leq [W]_{RH_{t,s}}
\]
and 
\[
    [AW^{Sc}_{Q,t}]_{A_\infty} \leq [W]^{Sc}_{A_t}\lesssim_{n,r} [W]_{A_t}.
\]
Thus we get
\[
    \vartheta^{-\frac{1}{r'}}\lesssim_d  \left([AW^{Sc}_{Q,t}]_{RH_{s}}[AW^{Sc}_{Q,t}]_{A_\infty}\right)^\frac{s}{r'} \lesssim_{n,q,r} \left([W]_{RH_{t,s}}[W]_{A_t}\right)^\frac{s}{r'}.
\]

Combining everything together we get that, assuming $1<p<r<q<\infty$ and $W\in A_t\cap RH_{t,s}$ where  $t\coloneqq\frac{r}{p}$ and $s\coloneqq\left(\frac{q}{r}\right)'$, we have
\begin{align*}
    \|T\vec f\|_{L^r_{B_2^n}(W)}&\lesssim_{d,n,p,q,r,\gamma}C_1 \,[W]_{A_t}^{\frac{t'}{r}+\frac{s}{r'}}[W]_{RH_{t,s}}^{\frac{1}{r}+\frac{s}{r'}} \|\vec f\|_{L^r_{B_1^n}(W)}.
\end{align*}
Note that if we set $p=1$ and $q=\infty$ in the above inequality, then we get \eqref{poneqinftynormest}. Furthermore, the same proof with obvious changes gives the same inequality with $p=1, q<\infty$ and $p>1, q=\infty$.
The results of the above discussion are summarized in the following theorem.
\begin{thm}\label{pqWeightnormineq}
    Let $1\leq p<r<q\leq\infty$, $t=\frac{r}{p}$ , $s=\left(\frac{q}{r}\right)'$ and $W\in A_t\cap RH_{t,s}$. Let $T$ be an operator that satisfies the convex body domination estimate
    \[
        |\langle T\vec f,\vec g\,\rangle|\leq C_1 \sum_{Q\in\mathscr Q}|Q|\llangle\Vec{f}\rrangle_{\textit{\L}^p(Q,B_1)}\cdot\llangle \vec g\rrangle_{\textit{\L}^{q'}(Q,B_2^*)},
    \]
    where $\mathscr Q$ is a $\gamma$-sparse family of cubes and $C_1$ is a constant. Then $T$ satisfies a matrix weighted norm inequality
    \[
        \|T\vec f\|_{L^r_{B_2^n}(W)}\leq C[W]_{A_t}^{\frac{t'}{r}+\frac{s}{r'}}[W]_{RH_{t,s}}^{\frac{1}{r}+\frac{s}{r'}} \|\vec f\|_{L^r_{B_1^n}(W)},
    \]
    where $C=C_1C_2$ and $C_2$ depends only on $d,n,p,q,r$ and $\gamma$.
\end{thm}
To the author's best knowledge the above is the first matrix-weighted norm inequality that has arisen from $(\textit{\L}^p,\textit{\L}^{q'})$ convex body domination at this level of generality with general exponents $1\leq p<q\leq\infty$. See Section 6 in \cite{di_plinio_sparse_2021} for results in the special case, where one can take $p=q'=1+\varepsilon$ for small enough $\varepsilon>0$.  

Note that $\frac{1}{r}+\frac{s}{r'}=\frac{q-1}{q-r}$ and $\frac{t'}{r}+\frac{s}{r'}=\frac{p}{(r-p)r}+\frac{q-1}{q-r}$. This modifies the conclusion of the above theorem to a form that is somewhat close to the scalar-valued version from \cite[Proposition 6.4]{bernicot_sharp_2016}. The scalar-valued conclusion is that the needed weight constant is
\[
    ([w]_{A_t}[w]_{RH_s})^{\max\{\frac{1}{r-p},\frac{q-1}{q-r}\}}
\]
and the exponent is sharp. Unfortunately, in most cases our result cannot reach the same exponent.

Theorem \ref{mainresult} combined with Theorem \ref{pqWeightnormineq} yields the following corollary. 
\begin{cor}\label{TweightCor}
    Let $p,r,q,t$ and $s$ be as in Theorem \ref{pqWeightnormineq}. If  $T=\sum_{j=N_1}^{N_2}T_j$ is a BRS operator, i.e., $T$ satisfies conditions \ref{supportCond}--\ref{restStrongLqbound} from Definition \ref{BRSop}, then we have 
    \[
        \|T\vec f\|_{L^r_{B_2^n}(W)}\lesssim_{d,n,p,q,r,\gamma,\kappa} \mathcal{C}\,[W]_{A_t}^{\frac{t'}{r}+\frac{s}{r'}}[W]_{RH_{t,s}}^{\frac{1}{r}+\frac{s}{r'}} \|\vec f\|_{L^r_{B_1^n}(W)}.
    \]
\end{cor}

\section{Commutators}
In this section we note how convex body domination for an operator $T$ implies sparse domination for the commutator of $T$. It is known that pointwise convex body domination of the operator $T$ implies certain pointwise convex body domination bounds for the commutator of $T$. See \cite[Theorem 4]{isralowitz_sharp_2021} and \cite[Theorem 1.8]{isralowitz_commutators_2022} for more details. We will not pursue such results in our setting, but instead we will prove that our definition of convex body domination for $T$ implies matrix-weighted bounds also for the commutator of $T$.

We assume that $B$ is a locally integrable matrix-valued function. Then the commutator $[B,T]$ is defined by
\[
    [B,T]\vec f\coloneqq BT\vec f-T[B\vec f].
\]
We will denote the average of $B$ over $Q$ by $\langle B\rangle_Q$ and define 
\[
    \mathcal{A}_\mathscr Q(\vec f,\vec g)\coloneqq \sum_{Q\in\mathscr Q}|Q|\left(\fint_Q |B(x)-\langle B\rangle_Q|_{op}^p|\vec f(x)|_{B_1^n}^p\intD x\right)^\frac{1}{p}\left(\fint_Q |\vec g(y)|_{B_2^*}^{q'}\intD y\right)^\frac{1}{q'}
\]
and 
\[
    \mathcal{A}^*_\mathscr Q(\vec f,\vec g)\coloneqq \sum_{Q\in\mathscr Q}|Q|\left(\fint_Q |B(x)-\langle B\rangle_Q|_{op}^{q'}|\vec g(y)|_{(B_2^*)^n}^{q'}\intD y\right)^\frac{1}{q'}\left(\fint_Q |\vec f(x)|_{B_2^*}^{p}\intD x\right)^\frac{1}{p}.
\]
The above notations and the following theorem are inspired by \cite[Theorem 1.1]{lerner_pointwise_2017}.
\begin{prop}\label{CommutatorSparse}
    Assume that $T$ satisfies $(\avgL^p_{B_1},\avgL^{q'}_{B_2^*})$ convex body domination, then for some sparse collection of dyadic cubes $\mathscr Q$ we have
    \[
        \langle [B,T]\vec f, \vec g\rangle\lesssim \mathcal{A}_\mathscr Q(\vec f,\vec g)+\mathcal{A}^*_\mathscr Q(\vec f,\vec g).
    \]
\end{prop}
\begin{proof}
We denote $\vec F \coloneqq\begin{bmatrix}\vec f\\B\vec f\end{bmatrix}$ and $\vec G\coloneqq\begin{bmatrix}
        B^\top \vec g\\-\vec g
    \end{bmatrix}$.
Recalling \eqref{extendbiliform} and \eqref{matrixswapbili}, we have 
\[
    \langle [B,T]\vec f, \vec g\rangle=
    \langle T\vec f, B^\top\vec g\rangle-\langle T[B\vec f],\vec g\rangle = \langle T\vec F,\vec G\rangle.
\]
Thus from the $(\avgL^p_{B_1},\avgL^{q'}_{B_2^*})$ convex body domination we get that
\begin{align*}
    \langle [B,T]\vec f, \vec g\rangle &\leq C \sum_{Q\in\mathscr Q} |Q|\llangle\vec F\rrangle_{\avgL^p(Q,B_1)}\cdot \llangle \vec G\rrangle_{\avgL^{q'}(Q,B_2^*)}.
\end{align*}
We estimate the convex bodies via Lemma \ref{convBodyCalc} and get
\begin{align*}
    (1+\varepsilon)^{-2}\llangle\vec F\rrangle_{\avgL^p(Q,B_1)}\cdot \llangle \vec G &\rrangle_{\avgL^{q'}(Q,B_2^*)}\leq\\&\left|\fint_Q\fint_Q \langle \vec F(x),\phi(x)\rangle_{(B_1,B_1^*)}\cdot\langle\vec G(y),\psi(y)\rangle_{(B_2^*,B_2^{**})}\intD x\intD y\right|,
\end{align*}
for some $\phi\in \Bar{B}_{\textit{\L}^{p'}_{B_1^*}(Q)}$ and $\psi\in\Bar{B}_{ \textit{\L}^{q}_{B_2^{**}(Q)}}$.
A simple calculation shows that the integrand is equal to
\begin{align*}
    \Big(B(y)-B(x)\Big)\langle \vec f(x),\phi(x)\rangle_{(B_1,B_1^*)}\cdot\langle \vec g(y),\psi(y)\rangle_{(B_2^*,B_2^{**})}.
\end{align*}
Thus we have
\begin{align*}
    (1+\varepsilon)^{-2}&\llangle\vec F\rrangle_{\avgL^p(Q,B_1)}\cdot \,\llangle \vec G\rrangle_{\avgL^{q'}(Q,B_2^*)}\\&\leq \fint_Q\fint_Q|B(y)-B(x)|_{op}|\langle \vec f(x),\phi(x)\rangle_{(B_1,B_1^*)}\cdot\langle \vec g(y),\psi(y)\rangle_{(B_2^*,B_2^{**})}|\intD x\intD y \\&\leq \fint_Q\fint_Q|B(y)-B(x)|_{op}|\vec f(x)|_{B_1}|\phi(x)|_{B_1^*}|\vec g(y)|_{B_2^*}|\psi(y)|_{B_2^{**}}\intD x\intD y.
\end{align*} 
Then we apply triangle inequality and Hölder's inequality to see that the last expression is bounded by
\begin{align*}
    \left(\fint_Q  |\vec f(x)|_{B_1}^p\intD x\right)^\frac{1}{p}&\left(\fint_Q |B(y)-\langle B\rangle_Q|_{op}^{q'}|\vec g(y)|_{B_2^*}^{q'}\intD x\right)^\frac{1}{q'} \\+&\left(\fint_Q |\vec g(y)|_{B_2^*}^{q'}\intD y\right)^\frac{1}{q'}\left(\fint_Q |B(x)-\langle B\rangle_Q|_{op}|^p\vec f(x)|_{B_1}^p\intD x\right)^\frac{1}{p}.
\end{align*}
Then letting $\varepsilon\to0$ gives us the result.
\end{proof}
We will then assume that the component functions $B_{ij}$ of $B$ are in $\operatorname{BMO}$. We find the $\operatorname{BMO}$ class useful due to the following well known result that follows by the John-Nirenberg inequality and Stirling approximation.
\begin{lem}\label{JNlemma}
    Let $1< a<\infty$ and $b\in \operatorname{BMO}$. Then we have 
    \[
        \left(\fint_Q|b(x)-\langle b\rangle_Q|^a\intD x\right)^\frac{1}{a}\lesssim_{d} a\,\|b\|_{BMO}.
    \]
\end{lem}
Now we are ready to show the sparse bound for the commutator $[B,T]$.
By the properties of the standard orthonormal basis of $\R^n$ we have
\[
    |B(x)-\langle B\rangle_Q|_{op}\eqsim_{n} \max_{i,j}|B_{ij}(x)-\langle B_{ij}\rangle_Q|.
\]
Then applying the Hölder inequality and Lemma \ref{JNlemma} we get that for any $u>1$ it holds that
\begin{align*}
    &\left(\fint_Q |B(x)-\langle B\rangle_Q|_{op}^p|\vec f(x)|_{B_1^n}^p\intD x\right)^\frac{1}{p}\\&\qquad\leq
    \left(\fint_Q |B(x)-\langle B\rangle_Q|_{op}^{pu'}\intD x\right)^\frac{1}{pu'}\left(\fint_Q|\vec f(x)|_{B_1^n}^{pu}\intD x\right)^\frac{1}{pu}
    \\&\qquad\eqsim_n \max_{i,j}
    \left(\fint_Q |B_{ij}(x)-\langle B_{ij}\rangle_Q|^{pu'}\intD x\right)^\frac{1}{pu'}\left(\fint_Q|\vec f(x)|_{B_1^n}^{pu}\intD x\right)^\frac{1}{pu}
    \\&\qquad\lesssim_{d,p}u'\max_{i,j}
    \|B_{ij}\|_{BMO} \left(\fint_Q|\vec f(x)|_{B_1^n}^{pu}\intD x\right)^\frac{1}{pu}.
\end{align*}
Similarly for the other term we have 
\begin{align*}
    &\left(\fint_Q |B(y)-\langle B\rangle_Q|_{op}^{q'}|\vec g(y)|_{(B_2^*)^n}^{q'}\intD y\right)^\frac{1}{q'}\\&\qquad\qquad\qquad\lesssim_{d,n,q} 
    u'\max_{i,j}\|B_{ij}\|_{BMO} \left(\fint_Q|\vec g(y)|_{(B_2^*)^n}^{q'u}\intD x\right)^\frac{1}{q'u}
\end{align*}
and thus by Proposition \ref{CommutatorSparse} we get a sparse domination inequality
\[
    \langle [B,T]\vec f, \vec g\rangle\lesssim_{d,n,p,q,u}\max_{i,j}\|B_{ij}\|_{BMO}\sum_{Q\in\mathscr Q}|Q|\|\vec f\|_{\avgL^{pu}(Q,B_1^n)}\|\vec g\|_{\avgL^{q'u}(Q,(B_2^*)^n)},\numberthis\label{sparseDomCom}
\]
where the dependence on $u$ is $u'$. We note that for the multi-scale operator that satisfies \ref{supportCond}--\ref{restStrongLqbound}, the sparse domination of commutators is new even in the scalar case. It is well documented, see \cite[Proposition 6.4]{bernicot_sharp_2016}, that with $pu<r<(q'u)'$ and $w\in A_{\frac{r}{pu}}\cap RH_{(\frac{(q'u)'}{r})'}$ the sparse bound \eqref{sparseDomCom} implies a norm inequality in $L^r(w)$ for the commutator of $T$. If $u$ is close enough to one, then by the self improvement properties of the Muckenhoupt and the reverse Hölder classes, it even suffices to take  $w\in A_{\frac{r}{p}}\cap RH_{(\frac{q}{r})'}$. However, the dependence on $u$ will increase the exponent of the weight constant. We will see that these ideas can be extended to the matrix-weighted case.

Our next goal is to adapt the proofs of Theorem \ref{pqWeightnormineq} and Proposition \ref{CommutatorSparse} to get a matrix-weighted norm inequality for the commutator $[B,T]$. To do this we will need the following elementary result.
\begin{lem}\label{ComNumbersLemma}
    Let $1<r<q<\infty$, $0<\vartheta<\frac{1}{2}$, $s=\left(\frac{q}{r}\right)'=\frac{q}{q-r}$ and $u=1+C_{q,r}\vartheta$ with $C_{q,r}=\frac{q-r}{r(q-1)+q(r-1)}$. Then
    \[
        r'-q'u\left(\frac{rs(1+\vartheta)}{q'u}\right)'\gtrsim_{q,r} \vartheta.
    \]
\end{lem}
\begin{proof}
    We start by calculating that
    \[
        q'u\left(\frac{rs(1+\vartheta)}{q'u}\right)'=r'\frac{(r-1)q(1+\vartheta)u}{r(q-1)(1+\vartheta)-u(q-r)}.
    \]
    Thus, we get
    \begin{align*}
        r'-q'u\left(\frac{rs(1+\vartheta)}{q'u}\right)'&=r'\left(1-\frac{(r-1)q(1+\vartheta)u}{r(q-1)(1+\vartheta)-u(q-r)}\right)\\&\eqsim_r 1-\frac{(r-1)q(1+\vartheta)(1+C_{q,r}\vartheta)}{r(q-1)(1+\vartheta)-(1+C_{q,r}\vartheta)(q-r)}\\&\eqqcolon\frac{N}{D} 
    \end{align*}
    The denominator $D$ is clearly less than $2r(q-1)$ and for the numerator $N$ we have
    \begin{align*}
        N&=r(q-1)(1+\vartheta)-(1+C_{q,r}\vartheta)(q-r)-(r-1)q(1+\vartheta)(1+C_{q,r}\vartheta)\\&=\vartheta C_{q,r}q(r-1)(1-\vartheta)\\&\gtrsim_{q,r}\vartheta.
    \end{align*}
    The second equality above is due to the definition of $C_{q,r}$.
\end{proof}
Now we are ready to show that convex body domination for the operator $T$ implies matrix-weighted bounds also for the commutator $[B,T]$. This result was inspired by similar results from \cite{isralowitz_sharp_2021} and \cite{isralowitz_commutators_2022}.
\begin{thm}\label{ComMatWghtNormIneq}
    Let $1\leq p<r<q\leq\infty$, $t=\frac{r}{p}$ , $s=\left(\frac{q}{r}\right)'$ and $W\in A_t\cap RH_{t,s}$. If $T$ satisfies $(\avgL^p_{B_1},\avgL^{q'}_{B_2^*})$ convex body domination and the components of the matrix $B$ are in $\operatorname{BMO}$, then 
    \begin{align*}
        \|[B,T]&\|_{L_{B_1^n}^r(W)\to L_{B_2^n}^r(W)}\leq \\&\qquad C\max_{i,j}\|B_{i,j}\|_{BMO} \left([W]_{A_t}^s[W]_{RH_{t,s}}^{s}+[W]_{A_t}^{\frac{1}{t-1}}\right)[W]_{A_t}^{\frac{t'}{r}+\frac{s}{r'}}[W]_{RH_{t,s}}^{\frac{1}{r}+\frac{s}{r'}},
    \end{align*}
    where $C=C_{n,d,p,q,r,\gamma}$ is a constant.
\end{thm}
\begin{proof}
     We denote $\vec f_{W}^{\,-}\coloneqq W^{-\frac{1}{r}}\vec f$, $\vec g_{W}^{\,+}\coloneqq W^\frac{1}{r}\vec g$ and 
    \[
        \mathscr{I}^{\phi,\psi}_{B,f,g}(x,y)\coloneqq |B(y)-B(x)|_{op}|\langle \vec f_{W}^{\,-}(x) ,\phi(x)\rangle_{(B_1,B_1^*)}\cdot\langle \vec g_{W}^{\,+}(y),\psi(y)\rangle_{(B_2^*,B_2^{**})}|,
    \] 
    where $\phi\in \Bar{B}_{\textit{\L}^{p'}_{B_1^*}(Q)}$ and $\psi\in\Bar{B}_{ \textit{\L}^{q}_{B_2^{**}(Q)}}$.
    Similarly as in the proof of Proposition \ref{CommutatorSparse}, for any $\varepsilon>0$ we have
    \begin{align*}
        \langle W^\frac{1}{r}[B,T](W^{-\frac{1}{r}}\vec f),\vec g\rangle&=\langle [B,T]\vec f_{W}^{\,-},\vec g_{W}^{\,+}\rangle\\&\lesssim (1+\varepsilon)^2\sum_{Q\in\mathscr Q}|Q|\fint_Q\fint_Q\mathscr{I}^{\phi,\psi}_{B,f,g}(x,y)\intD x\intD y.
    \end{align*}
    We note that $\mathscr{I}^{\phi,\psi}_{B,f,g}(x,y)$ is bounded from above by
    \[
          |B(y)-B(x)|_{op}\left|W^{-\frac{1}{r}}(x) W^\frac{1}{r}(y)\right|_{op}|\vec f(x)|_{B_1^n}|\phi(x)|_{B_1^*}|\vec g(y)|_{(B_2^*)^n}|\psi(y)|_{B_2^{**}}.
    \]
    Thus applying triangle inequality $|B(x)-B(y)|\leq |B(x)-\langle B\rangle_Q|+|B(y)-\langle B\rangle_Q|$ and Hölder's inequality yields
    \[
        \sum_{Q\in\mathscr Q}|Q|\fint_Q\fint_Q\mathscr{I}_{B,f,g}(x,y)\intD x\intD y\leq \mathcal{A}_{\mathscr Q, W}(\vec f,\vec g)+\mathcal{A}_{\mathscr Q, W}^*(\vec f,\vec g),
    \]
    where
    \begin{align*}
        \mathcal{A}_{\mathscr Q, W}(\vec f,\vec g)\coloneqq \sum_{Q\in\mathscr Q}|Q|\Bigg(\fint_Q &|B(y)-\langle B\rangle_Q|_{op}^{q'}|\vec g(y)|_{(B_2^*)^n}^{q'}\\&\left[\fint_Q\left|W^{-\frac{1}{r}}(x) W^\frac{1}{r}(y)\right|_{op}^p|\vec f(x)|^p_{B_1^n}\intD x\right]^\frac{q'}{p} \intD y\Bigg)^\frac{1}{q'}
    \end{align*}
    and 
    \begin{align*}
        \mathcal{A}_{\mathscr Q, W}^*(\vec f,\vec g)\coloneqq \sum_{Q\in\mathscr Q}|Q|\Bigg(\fint_Q &|B(x)-\langle B\rangle_Q|_{op}^{p}|\vec f(x)|_{B_1^n}^{p}\\&\left[\fint_Q\left|W^{-\frac{1}{r}}(x) W^\frac{1}{r}(y)\right|_{op}^{q'}|\vec g(y)|^{q'}_{(B_2^*)^n}\intD y\right]^\frac{p}{q'} \intD x\Bigg)^\frac{1}{p}.
    \end{align*}
    Letting $\varepsilon\to0$ gives \begin{equation}\label{comFirstStep}
        \langle W^\frac{1}{r}[B,T](W^{-\frac{1}{r}}\vec f),\vec g\rangle\lesssim \mathcal{A}_{\mathscr Q, W}(\vec f,\vec g)+\mathcal{A}_{\mathscr Q, W}^*(\vec f,\vec g)
    \end{equation}
    and we are left with bounding $\mathcal{A}_{\mathscr Q,W}$ and $\mathcal{A}^*_{\mathscr Q,W}$.
    
    We will consider $\mathcal{A}_{\mathscr Q,W}$ first. The term inside the sum is dominated by
    \begin{align*}
    |Q|\left|A_{Q,t}^\frac{1}{p}B_{Q,t}^\frac{1}{p}\right|_{op}&\left(\fint_Q |\vec f(x)|^{p}_{B_1^n}\left|B^{-\frac{1}{p}}_{Q,t}W^{-\frac{1}{r}}(x)\right|_{op}^{p}\intD x \right)^\frac{1}{p}\\&\qquad\qquad\left(\fint_Q |B(y)-\langle B\rangle_Q|^{q'}_{op}|\vec g(y)|^{q'}_{(B_2^*)^n}\left|A^{-\frac{1}{p}}_{Q,t}W^\frac{1}{r}(y)\right|_{op}^{q'}\intD y \right)^\frac{1}{q'},
\end{align*}
where $A_{Q,t}$ and $B_{Q,t}$ are the reducing matrices that we defined in the beginning of Section 6. Similarly as in the proof of Theorem \ref{pqWeightnormineq} we apply the Cordes inequality (Lemma \ref{Cordesineq}) and the Hölder inequality to get  
\begin{align*}
    \mathcal{A}_{\mathscr Q,W}(\vec f,\vec g)\lesssim [W]_{A_t}^\frac{1}{r}S_{\vec f}\,\tilde S_{\vec g,B},
\end{align*}
where
\[
    S_{\vec f}\coloneqq \left(\sum_{Q\in\mathscr Q}|Q|\left[\fint_Q |\vec f(x)|^{p}_{B_1^n}\left|B^{-1}_{Q,t}W^{-\frac{1}{t}}(x)\right|_{op}\intD x \right]^\frac{r}{p}\right)^\frac{1}{r}
\]
and 
\[
    \tilde S_{\vec g,B} \coloneqq \left(\sum_{Q\in\mathscr Q}|Q|\left[\fint_Q |B(y)-\langle B\rangle_Q|^{q'}_{op}|\vec g(y)|^{q'}_{(B_2^*)^n}\left|A^{-1}_{Q,t}W^{\frac{1}{t}}(x)\right|_{op}^\frac{q'}{p}\intD x \right]^\frac{r'}{q'}\right)^\frac{1}{r'}.
\]
In the proof of Theorem \ref{pqWeightnormineq} we showed that
\[
    S_{\vec f} \lesssim_{d,n,p,r,\gamma} [W]_{A_t}^\frac{1}{r(t-1)}\|\vec f\|_{L^r_{B_1^n}}.
\]
In order to bound $\tilde S_{\vec g,B}$ we let \[\vartheta\coloneqq\frac{1}{2^{d+1}[(AW^{Sc}_{Q,t})^s]_{A_\infty}-1} \quad\text{ with } \quad AW^{Sc}_{Q,t}(y)\coloneqq\left|A^{-1}_{Q,t}W^\frac{1}{t}(y)\right|_{op}^{t}\]
and $u\coloneqq1+C_{q,r}\vartheta$ with $C_{q,r}\coloneqq\frac{q-r}{r(q-1)+q(r-1)}$.
Applying the Hölder inequality and Lemma \ref{JNlemma} to the term inside the sum in $\tilde S_{\vec g,B}$ we get
\begin{align*}
    &\left[\fint_Q |B(y)-\langle B\rangle_Q|^{q'}_{op}|\vec g(y)|^{q'}_{(B_2^*)^n}\left|A^{-1}_{Q,t}W^\frac{1}{t}(y)\right|_{op}^\frac{q'}{p}\intD y \right]^\frac{1}{q'}\\&\qquad\quad\lesssim_{d,n,q} u'\max_{i,j}\|B_{i,j}\|_{BMO} \left[\fint_Q |\vec g(y)|^{q'u}_{(B_2^*)^n}\left|A^{-\frac{1}{p}}_{Q,t}W^\frac{1}{r}(y)\right|_{op}^\frac{q'u}{p}\intD y \right]^\frac{1}{q'u}.
\end{align*}
 Furthermore, Lemma \ref{scarevHölrmrk} implies that
\[  
    u'\lesssim_{q,r}\vartheta^{-1} \lesssim_{d,n,p,q,r} ([W]_{RH_{t,s}}[W]_{A_t})^s.
\]
Then we apply Hölder's inequality to get 
\begin{align*}
    \left[\fint_Q |\vec g(y)|^{q'u}_{(B_2^*)^n}\left|A^{-\frac{1}{p}}_{Q,t}W^\frac{1}{r}(y)\right|_{op}^\frac{q'u}{p}\intD y \right]^\frac{1}{q'u}\leq \left(\fint_Q \left|A^{-1}_{Q,t}W^\frac{1}{t}(y)\right|_{op}^{ts(1+\vartheta)}\intD y \right)^{\frac{1}{ts(1+\vartheta)}\frac{1}{p}}&\\\left(\fint_Q |\vec g(y)|^{b}_{(B_2^*)^n}\intD y\right)^\frac{1}{b},
\end{align*}
where
\[
    b\coloneqq q'u\left(\frac{pts(1+\vartheta)}{q'u}\right)'=q'u\left(\frac{rs(1+\vartheta)}{q'u}\right)'.
\]
We already showed in the proof of Theorem \ref{pqWeightnormineq} that
\[  
    \left(\fint_Q \left|A^{-1}_{Q,t}W^\frac{1}{t}(y)\right|_{op}^{ts(1+\vartheta)}\intD y \right)^{\frac{1}{ts(1+\vartheta)}\frac{1}{p}}\lesssim_n[W]^\frac{1}{t}_{RH_{t,s}}
\]
and by Lemma \ref{ComNumbersLemma} we have
\[
    \sum_{Q\in\mathscr Q}\left(\fint_Q |\vec g(y)|^{b}_{(B_2^*)^n}\intD y\right)^\frac{1}{b}\lesssim_{\gamma} \left(\frac{r'}{r'-b}\right)^\frac{1}{b}\|\vec g\|_{L^{r'}_{(B_2^*)^n}}\lesssim_{q,r} \vartheta^{-\frac{1}{b}}\|\vec g\|_{L^{r'}_{(B_2^*)^n}}.
\]
Furthermore, applying Lemmas \ref{numbersLemma} and \ref{scarevHölrmrk} yields
\[
    \vartheta^{-\frac{1}{b}}\leq \vartheta^{-\frac{1}{q'u\left(\frac{rs}{q'u}\right)'}}\leq \vartheta^{-\frac{1}{q'\left(\frac{rs}{q'}\right)'}}=\vartheta^{-\frac{1}{r'}}\lesssim_{d,n,p,q,r} \left([W]_{RH_{t,s}}[W]_{A_t}\right)^\frac{s}{r'}.
\]
We have now shown that 
\[
    \mathcal{A}_{\mathscr Q,W}(\vec f,\vec g)\lesssim_{d,n,p,q,r,\gamma}\max_{i,j}\|B_{i,j}\|_{BMO} [W]_{A_t}^{s+\frac{t'}{r}+\frac{s}{r'}}[W]_{RH_{t,s}}^{s+\frac{1}{r}+\frac{s}{r'}}\|\vec g\|_{L^{r'}_{(B_2^*)^n}}\|\vec f\|_{L^r_{B_1^n}}.
\]
Similar calculations shows that for the term $\mathcal{A}_{\mathscr Q,W}^*$ we have
\[
    \mathcal{A}^*_{\mathscr Q,W}(\vec f,\vec g)\lesssim_{d,n,p,q,r,\gamma}\max_{i,j}\|B_{i,j}\|_{BMO} [W]_{A_t}^{\frac{1}{t-1}+\frac{t'}{r}+\frac{s}{r'}}[W]_{RH_{t,s}}^{\frac{1}{r}+\frac{s}{r'}}\|\vec g\|_{L^{r'}_{(B_2^*)^n}}\|\vec f\|_{L^r_{B_1^n}}.
\]
Combining the above bounds with \eqref{comFirstStep} we get
\begin{align*}
    \langle& W^\frac{1}{r}[B,T](W^{-\frac{1}{r}}\vec f),\vec g\rangle\lesssim_{d,n,p,q,r,\gamma}\\&\max_{i,j}\|B_{i,j}\|_{BMO} \|\vec f\|_{L^r_{B_1^n}}\|\vec g\|_{L^{r'}_{(B_2^*)^n}}\left([W]_{A_t}^s[W]_{RH_{t,s}}^{s}+[W]_{A_t}^{\frac{1}{t-1}}\right)[W]_{A_t}^{\frac{t'}{r}+\frac{s}{r'}}[W]_{RH_{t,s}}^{\frac{1}{r}+\frac{s}{r'}}
\end{align*}
and duality completes the proof.
\end{proof}
Theorem \ref{mainresult} combined with Theorem \ref{ComMatWghtNormIneq} yields the following corollary.
\begin{cor}\label{ComTweightCor}
    Let $p,r,q,t$ and $s$ be as in Theorem \ref{ComMatWghtNormIneq}. If  $T=\sum_{j=N_1}^{N_2}T_j$ is a BRS operator, i.e., $T$ satisfies conditions \ref{supportCond}--\ref{restStrongLqbound} from Definition \ref{BRSop}, then we have
    \begin{align*}
        \|[B,T]&\|_{L^r_{B_1^n}(W)\to L^r_{B_2^n}(W)}\lesssim_{d,n,p,q,r,\gamma,\kappa} \\&\qquad\qquad\mathcal{C}\,\max_{i,j}\|B_{i,j}\|_{BMO} \left([W]_{A_t}^s[W]_{RH_{t,s}}^{s}+[W]_{A_t}^{\frac{1}{t-1}}\right)[W]_{A_t}^{\frac{t'}{r}+\frac{s}{r'}}[W]_{RH_{t,s}}^{\frac{1}{r}+\frac{s}{r'}}.
    \end{align*}
\end{cor}

\section{From dense subspaces to global bounds}
In this section we show that in order to prove convex body domination bounds in the Bochner spaces $L^r_{B_1}$ and $L^{r'}_{B_2^*}$, it suffices to prove the same bounds in some dense subspaces. The proof is once again an adaptation of the argument from the scalar valued case, see \cite[Lemma A.1.]{beltran_multi-scale_2020}.
\begin{prop}\label{denseToGlobal}
    Suppose $1\leq p<r<q\leq\infty$, and let $V_1\subset L^r_{B_1}$ and $V_2\subset L^{r'}_{B_2^*}$ be dense subspaces. Then the $(\textit{\L}^p_{B_1},\textit{\L}^{q'}_{B_2^*})$ convex body domination bound extends from $V_1\times V_2$ to $L^r_{B_1}\times L^{r'}_{B_2^*}$.
\end{prop}
\begin{proof}
    Let $\vec f\in L^r_{B_1^n}$ and $\vec g\in L^{r'}_{(B_2^*)^n}$. To simplify notation we will denote the Bochner $L^p$ norm simply with $\|\cdot\|_p$. We will also denote
    \[
        \Lambda_{p,q}^{\mathscr Q}(\vec f,\vec g)\coloneqq \sum_{Q\in \mathscr Q} |Q|\llangle \vec f\rrangle_{\textit{\L}^p(Q,B_1)}\cdot\llangle \vec g \rrangle_{\textit{\L}^{q'}(Q,B_2^*)},
    \]
    where $\mathscr Q$ is some $\gamma$-sparse collection.
    From the weighted bounds with $W=Id_n$ we see that 
    \begin{align*}
        \Lambda_{p,q}^{\mathscr Q}(\vec f,\vec g)\lesssim_{p,q,r,\gamma}\|\vec f\|_r\|\vec g\|_{r'}.\numberthis\label{TboundedbyCBD}
    \end{align*}
    Thus the assumed convex body domination implies that the bilinear form associated to $T$ is bounded on $V_1^n\times V_2^n$ and hence there exists a unique extension that is bounded on $L^r_{B_1^n}\times L^{r'}_{(B^*_2)^n}$.
    For any $\varepsilon>0$, we choose $\vec f_*\in V_1^n$ and $\vec g_*\in V_2^n$ such that  $\|\vec f-\vec f_*\|_{r}\leq \varepsilon$ and $\|\vec g-\vec g_*\|_{r'}\leq \varepsilon$. The choice can be made so that we also have $\|\vec f_*\|_{r}\leq 2\|\vec f\|_{r}$ and $\|\vec g_*\|_{r'}\leq 2\|\vec g\|_{r'}$. 
    We need to show that 
    \[
        |\langle T \vec f_*, \vec g_*\rangle|
        \lesssim \Lambda_{p,q}^{\mathscr Q}(\vec f_*,\vec g_*)\quad\Rightarrow\quad|\langle T \vec f, \vec g\,\rangle|
        \lesssim \Lambda_{p,q}^{\mathscr Q}(\vec f,\vec g).
    \]
    We start by estimating
    \begin{align*}
        |\langle T \vec f, \vec g\,\rangle| &\leq |\langle T [\vec f-\vec f_*], \vec g\,\rangle|+|\langle T \vec f_*, \vec g-\vec g_*\rangle|+|\langle T \vec f_*, \vec g_*\rangle|\\&\leq
        \|T\|_{r\to r}\left(\|\vec f-\vec f_*\|_r\|\vec g\|_{r'}+\|\vec f_*\|_r\|\vec g-\vec g_*\|_{r'}\right) + |\langle T \vec f_*, \vec g_*\rangle|
        \\&\leq \varepsilon\,\|T\|_{r\to r}\left(\|\vec g\|_{r'}+2\|\vec f\|_r\right) + |\langle T \vec f_*, \vec g_*\rangle|.
    \end{align*}
    By assumption we have
    \[
        |\langle T \vec f_*, \vec g_*\rangle|\lesssim \Lambda_{p,q}^{\mathscr Q}(\vec f_*,\vec g_*).
    \]
    Then we have the identity
    \begin{align*}
        \llangle \vec f_*\rrangle_{\textit{\L}^p(Q,B_1)}\cdot\llangle \vec g_* \rrangle_{\textit{\L}^{q'}(Q,B_2^*)}=&\,\,\llangle \vec f_*-\vec f\rrangle_{\textit{\L}^p(Q,B_1)}\cdot\llangle \vec g_* \rrangle_{\textit{\L}^{q'}(Q,B_2^*)}\\&+\llangle \vec f\rrangle_{\textit{\L}^p(Q,B_1)}\cdot\llangle \vec g_* -\vec g\rrangle_{\textit{\L}^{q'}(Q,B_2^*)}\\&+\llangle \vec f\rrangle_{\textit{\L}^p(Q,B_1)}\cdot\llangle \vec g \rrangle_{\textit{\L}^{q'}(Q,B_2^*)}
    \end{align*}
    and this implies that
    \[
        \Lambda_{p,q}^{\mathscr Q}(\vec f_*,\vec g_*) = \Lambda_{p,q}^{\mathscr Q}(\vec f_*-\vec f,\vec g_*) + \Lambda_{p,q}^{\mathscr Q}(\vec f,\vec g_*-\vec g)+ \Lambda_{p,q}^{\mathscr Q}(\vec f,\vec g).
    \]
     By \eqref{TboundedbyCBD} we have
    \begin{align*}
        \Lambda_{p,q}^{\mathscr Q}(\vec f_*-\vec f,\vec g_*)+ \Lambda_{p,q}^{\mathscr Q}(\vec f,\vec g_*-\vec g)&\lesssim_{p,q,r,\gamma} \|\vec f_*-\vec f\|_r\|\vec g_*\|_{r'}+\|\vec f\|_r\|\|\vec g_*-\vec g\|_{r'}\\&\leq \varepsilon\left(2\|\vec g\|_{r'}+\|\vec f\|_r\right).
    \end{align*}
    We have now shown that
    \begin{align*}
        |\langle T \vec f, \vec g\,\rangle|&\lesssim \varepsilon\,\|T\|_{r\to r}\left(\|\vec g\|_{r'}+\|\vec f\|_r\right) + \Lambda_{p,q}^{\mathscr Q}(\vec f_*,\vec g_*)\\
        &\lesssim_{p,q,r,\gamma} \varepsilon\,\big(\|T\|_{r\to r}+1\big)\big(\|\vec g\|_{r'}+\|\vec f\|_r\big) + \Lambda_{p,q}^{\mathscr Q}(\vec f,\vec g).
    \end{align*}
    Letting $\varepsilon\to0$ in the above inequality gives us
    \[  
        |\langle T \vec f, \vec g\,\rangle|\lesssim_{p,q,r,\gamma} \Lambda_{p,q}^{\mathscr Q}(\vec f,\vec g)
    \]
    and we are done.
\end{proof}
The above result lets us extend the results that we have proven in sections 4 and 5  for $V_1=S_{B_1}$, $V_2=S_{B_2^*}$ to $L^r_{B_1}\times L^r_{B_2^*}$.

\section{Fourier multipliers}
In this section we will note  that certain Fourier multiplier operators satisfy the conditions \ref{supportCond}--\ref{restStrongLqbound} and hence the convex body estimate of Theorem \ref{mainresult} is true for these operators. The proof is almost identical to the one in \cite{beltran_multi-scale_2020} Section 6.3 so we only give an outline of the proof. For more details see the proof in \cite{beltran_multi-scale_2020}.

Let $\mathcal H_1$ and $\mathcal H_2$ be two separable Hilbert spaces and denote the space of $(\mathcal H_1,\mathcal H_2)$-bounded linear operators by $\mathcal{L}(\mathcal H_1,\mathcal H_2)$. We consider the multiplier operator $T=T_m$ that maps $\mathcal H_1$-valued functions to $\mathcal H_2$-valued functions and is defined on the frequency side by
\[
    \widehat {Tf}(\xi)=m(\xi)\widehat f(\xi),
\]
where $m$ is an operator-valued function such that $m(\xi)\in \mathcal{L}(\mathcal H_1,\mathcal H_2)$ for almost every $\xi$. For $1\leq p\leq \infty$ we define the operator-valued $L^p$ space 
\[
    L^p_{\mathcal{L}(\mathcal H_1,\mathcal H_2)}\coloneqq\{m\,\colon\,mx\in L^p_{\mathcal{H}_2}, \forall x\in \mathcal{H}_1\}
\]
that has the norm 
\[
    \|m\|_{ L^p_{\mathcal{L}(\mathcal H_1,\mathcal H_2)}}\coloneqq \sup_{|x|_{\mathcal{H}_1}\leq1}\|mx\|_{L^p_{\mathcal H_2}}.
\]
Also for $1\leq p\leq q\leq \infty$ we define $ M^{p,q}_{\mathcal H_1,\mathcal H_2}$ to be the space of multipliers $m$ that satisfy 
\[
    \|Tf\|_{L^q_{\mathcal H_2}}\leq C \|f\|_{L^p_{\mathcal H_1}}.
\]
Furthermore, the optimal constant defines a norm in $M^{p,q}_{\mathcal H_1,\mathcal H_2}$. 

To see the multi-scale structure of $T$, we consider functions of the form 
\[\Psi_\ell(x)\coloneqq\Psi_0(2^{-\ell}x)-\Psi_0(2^{-\ell+1}x),\] where $\ell>0$ and $\Psi_0\in C^\infty_c$  with $\operatorname{supp}\Psi_0\subset \{x\in\R^d\colon |x|<\frac{1}{2}\}$ and $\Psi_0(x)=1$ for $|x|<\frac{1}{4}$. Let $\theta\in C^\infty_c$ such that $\operatorname{supp}\theta\subset\{x\in\R^d\colon |x|<\frac{1}{2}\}$ and $\int_{\R^d}\pi(x)\theta(x)\intD x=0$ for all polynomials $\pi$ of degree at most $10d$. We also take a radial $C^\infty$ function $\phi$ that is supported in $\{\xi\in \widehat\R^d\,\colon\, \frac{1}{2}<|\xi|<2\}$ and 
\[\sum_{k\in\Z}\phi(2^{-k}\xi)\widehat \theta(2^{-k}\xi)=1.\]

For now we will assume that $m$ is compactly supported in $\widehat \R^d\setminus\{0\}$. Then with $n_1,n_2\in \Z$ we can write 
\[
    m(\xi)= \sum_{k=n_1}^{n_2} \phi(2^{-k}\xi)\widehat \theta(2^{-k}\xi)m(\xi)
\]
and the multiplier operator can be decomposed into a multi-scale sum
\[
   Tf=\sum_{\ell\geq 0}\sum_{j=\ell+1-n_2}^{\ell+1-n_1}T_j^\ell f, 
\]
where $T^\ell_j$ is a convolution operator with kernel $K_{\ell+1-j}^\ell$ defined by
\[
    K_k^\ell(x)\coloneqq \int_{\R^d}\mathcal{F}^{-1}[\phi(2^{-k}\cdot)m](x-y)\Psi_\ell(2^k(x-y))2^{jd}\theta(2^ky)\intD y, \quad k\in\Z.
\]
We denote
\[
    \mathfrak a_{\circ,\ell}[m] \coloneqq \sup_{t>0}\|[\phi m(t\cdot)]*\widehat\Psi_\ell\|_{L^\infty_{\mathcal{L}(\mathcal{H}_1,\mathcal{H}_2)}},\quad B_\circ[m]\coloneqq\sum_{\ell\geq0} \mathfrak a_{\circ,\ell}[m].
\]
and 
\[
    \mathfrak a_{\ell}[m] \coloneqq \sup_{t>0}\|[\phi m(t\cdot)]*\widehat\Psi_\ell\|_{M^{p,q}_{\mathcal{H}_1,\mathcal{H}_2}}2^{\ell d\big(\frac{1}{p}-\frac{1}{q} \big)},\quad \mathcal B[m]\coloneqq\sum_{\ell\geq0}\mathfrak a_\ell[m](1+\ell).
\]
It was shown in \cite{beltran_multi-scale_2020} Section 6.3  that for $\ell\geq0$ the operator \[ T^\ell\coloneqq\sum_{j=\ell+1-n_2}^{\ell+1-n_1}T_j^\ell\] satisfies conditions \ref{supportCond}--\ref{restStrongLqbound} with \[ A(p)=A(q)\eqsim_{d,p,q} \mathfrak a_{\circ,\ell}[m]+\mathfrak a_\ell[m](1+\ell),\quad
A_\circ(p,q)\eqsim_{d,p,q} \mathfrak a_\ell[m]\] and 
\[ B\eqsim_{d,p,q} 2^{\kappa\ell}\mathfrak a_\ell[m].\]
Thus applying Theorem \ref{mainresult} and summing over $\ell\geq0$ we get that if $m$ is compactly supported, $1<p\leq q<\infty$ and
\[
    \mathcal{C}\coloneqq \mathcal B_\circ[m]+\mathcal B[m]<\infty,
\]
then 
\[
    \langle T \vec f,\vec g\,\rangle \lesssim_{d,n,p,q,\gamma} \mathcal{C}\sum_{Q\in\mathscr{Q}} |Q| \,\llangle\Vec{f}\rrangle_{\textit{\L}^p(Q,B_1)}\cdot \llangle\Vec{g}\rrangle_{\textit{\L}^{q'}(Q,B_2^*)}.\numberthis\label{FMconvboddom}
\]

In \cite[p. 63-64] {beltran_multi-scale_2020} it was shown how to remove the assumption that $m$ is compactly supported in the scalar case and the same argument with minor obvious changes works also in the vector case. Thus if $1<p\leq q<\infty$ and \[B_\circ[m]+\mathcal B[m]<\infty,\numberthis\label{multiplierCond}\] then we have the $(\textit{\L}^p_{B_1},\textit{\L}^{q'}_{B_2^*})$ convex body domination bound \eqref{FMconvboddom} with the general multiplier $m$. 

We refer to \cite[Chapters 6 and 7]{beltran_multi-scale_2020} for many concrete examples of Fourier multiplier operators that satisfy \eqref{multiplierCond}. Their list includes Miyachi classes with subdyadic Hörmander conditions, prototypical versions of singular Radon transforms, multi-scale variants of oscillatory multipliers and certain radial multipliers. By \eqref{FMconvboddom}, Corollary \ref{TweightCor} and Corollary \ref{ComTweightCor} all these operators and their commutators with BMO functions satisfy matrix-weighted norm inequalities.

\appendix
\section{Proof of the inequality (\ref{TheClaim})}
We will prove the claim that we made at the end of the proof of Lemma \ref{inductiveLemma}. We stress the fact that we follow the arguments from \cite{beltran_multi-scale_2020} and there is very little novelty in the proof of the claim. We will need the following dyadic version of the regularity conditions \ref{singleScalereqcond} and \ref{adjSingleScalereqcond} from Definition \ref{BRSop}.
\begin{lem}\label{BRScor} Let $T_j$ be a single scale operator that satisfies \ref{singleScalereqcond} and \ref{adjSingleScalereqcond} from Definition \ref{BRSop}.
Let $1< p\leq q<\infty$, $0<\vartheta<\frac{1}{p}$ and $\mathbb{E}_kf\coloneqq\sum_{Q}\mathbbm{1}_Q\fint_Qf$, where the sum is taken over dyadic cubes $Q$ with side length $2^{-k}$. Then
\[\|T_j(I-\mathbb{E}_{k-j})\|_{L^p_{B_1}\to L^q_{B_2}}\lesssim_\vartheta B2^{-k\vartheta}2^{-jd(\frac{1}{p}-\frac{1}{q})}.\]
and 
\[\|T_j^*(I-\mathbb{E}_{k-j})\|_{L^{q'}_{B_2^*}\to L^{p'}_{B_1^*}}\lesssim_\vartheta B2^{-k\vartheta}2^{-jd(\frac{1}{p}-\frac{1}{q})}.\]
\end{lem}
\noindent  See \cite[Corollary 3.5]{beltran_multi-scale_2020} for the proof of the above lemma.

We recall the setting of the claim. We have $f_i\coloneqq R_f\vec f\cdot \vec e_i$ and $g_i\coloneqq R_f^{-\top}\vec g\cdot \vec e_i$. Then  $\Omega\coloneqq\bigcup_{i=1}^n\Omega_i^{(f)}\cup \Omega_i^{(g)},$ where
\[
    \Omega_i^{(f)}\coloneqq\left\{x\in3Q_0\,\colon\,\mathcal{M}_pf_i(x)>\Big(\frac{100^dn}{1-\gamma}\Big)^\frac{1}{p}\| f_i\|_{\textit{\L}^p(Q_0,B_1)}\right\}
\] 
and
\[
    \Omega_i^{(g)}\coloneqq\left\{x\in3Q_0\,\colon\,\mathcal{M}_{q'}g_i(x)>\Big(\frac{100^dn}{1-\gamma}\Big)^\frac{1}{q'}\| g_i\|_{\textit{\L}^{q'}(3Q_0,B_2^*)}\right\}.
\]
We decompose $\Omega$ via a Whitney collection $\mathcal W$ that consists of disjoint cubes $P$ with side length $2^{L(P)}\in2^\Z$, such that 
\[
    5\operatorname{diam}(P)\leq \operatorname{dist}(P,\Omega^\complement)\leq 12\operatorname{diam}(P),\quad\text{for all } P\in\mathcal{W}.
\]  
Recall also that  
\[
    S_Qf\coloneqq\sum_{N_1\leq j\leq L(Q)}T_j[f\mathbbm{1}_Q]
\] 
for a dyadic cube $Q$ with side length $2^{L(Q)}\in2^\Z$.
\begin{prop}\label{appendixProp}
    Fix a dyadic cube $Q_0$ with side length $2^{N_2}$ and let $1< p\leq q<\infty$. Let  $f_i$, $g_i$, $\Omega$, $\mathcal{W}$ and $S_Q$ be as above.  Then for every $i=1,\dots,n$
     we have
    \begin{align*}|\langle Tf_i,g_i\rangle- \sum_{P\in\mathcal{W}}&\langle S_Pf_i,g_i\rangle|\\&\lesssim_{d,p,q,\gamma,\kappa} \mathcal{C}n^{\frac{1}{p}+\frac{1}{q'}}\,|3Q_0|\,\| f_i\|_{\textit{\L}^p(3Q_0,B_1)}\| g_i\|_{\textit{\L}^{q'}(3Q_0,B_2^*)}.\end{align*}
\end{prop}
\begin{proof}
    Recall that for $i=1,\dots,n$ we defined
    \begin{align*}
        b_{i,P}\coloneqq(f_i-\fint_P f_i)\mathbbm{1}_P,\qquad b_i\coloneqq\sum_{\substack{P\in\mathcal{W}\\P\subset Q_0}}b_{i,P},
    \end{align*}
    and
    \begin{align*}
        h_i\coloneqq f_i\mathbbm{1}_{\Omega^\complement}+\sum_{\substack{P\in\mathcal{W}\\P\subset Q_0}}\mathbbm{1}_P\fint_P f_i.
    \end{align*}
    We also need an analogous decomposition for $g_i$ so we define $b^*_{i,P}$, $b^*_{i}$ and $h_i^*$ similarly as above, but with $g_i$ instead of $f_i$. Now we have
    \[
        \| f_i\|_{\textit{\L}^p(P,B_1)}\lesssim_{d,\gamma}n^\frac{1}{p}\| f_i\|_{\textit{\L}^p(Q_0,B_1)}, \quad\| g_i\|_{\textit{\L}^{q'}(P,B_2^*)}\lesssim_{d,\gamma}n^\frac{1}{q'}\| g_i\|_{\textit{\L}^{q'}(3Q_0,B_2^*)}.\numberthis\label{arverageDilation}
    \]
    Indeed, since
    \[
        5\operatorname{diam}(P)\leq \operatorname{dist}(P,\Omega^\complement)\leq 12\operatorname{diam}(P),\quad\text{for all } P\in\mathcal{W},
    \]
    we can take a suitable dimensional dilate of $P$, denoted by $P_D$, such that the set $P_D\cap\Omega^\complement$ is non-empty. Then taking $x\in P_D\cap\Omega^\complement$ leads to
    \begin{align*}\| f_i\|_{\textit{\L}^p(P,B_1)}\lesssim_d \| f_i\|_{\textit{\L}^p(P_D,B_1)}&\leq \mathcal{M}_pf_i(x)\leq n^\frac{1}{p}\left(\frac{100^d}{1-\gamma}\right)^\frac{1}{p}\| f_i\|_{\textit{\L}^p(Q_0,B_1)}.\end{align*}
    A similar argument for $g_i$ gives us the other inequality from \eqref{arverageDilation}. Applying this and the definition of $\Omega$, we also get
    \[
        \|h_i\|_{L^\infty_{B_1}}\lesssim_{d,\gamma}n^\frac{1}{p}\| f_i\|_{\textit{\L}^p(Q_0,B_1)}, \quad \|h_i^*\|_{L^\infty_{B_2*}}\lesssim_{d,\gamma}n^\frac{1}{q'}\| g_i\|_{\textit{\L}^{q'}(3Q_0,B_2^*)}.\numberthis\label{goodIsInLinfty}
    \]
    Since $\operatorname{supp}(f_i) \subset Q_0$ and $\operatorname{supp}(g_i)\subset 3Q_0$ we also get $\operatorname{supp}(h_i) \subset Q_0$ and $\operatorname{supp}(h^*_i) \subset 3Q_0$. Since $ \|\mathbbm{1}_Q\|_{L^{r,1}}\lesssim_r|Q|^{1/r}$ for $r < \infty$, we get from \eqref{goodIsInLinfty} that for $r_1, r_2 < \infty$,
    \[                               \|h_i\|_{L^{r_1,1}_{B_1}}\lesssim_{d,r_1,\gamma}n^\frac{1}{p}|Q_0|^{\frac{1}{r_1}}\| f_i\|_{\textit{\L}^p(Q_0,B_1)}\numberthis\label{GoodIsRestrictedStrongType}
    \]
    and
    \[
    \|h_i^*\|_{L^{r_2,1}_{B_2^*}} \lesssim_{d,r_2,\gamma}n^\frac{1}{q'}|Q_0|^{\frac{1}{r_2}}\| g_i\|_{\textit{\L}^{q'}(3Q_0,B_2^*)}.\numberthis\label{GoodIsRestrictedStrongTypeStar}
    \]
    Note that for $r>1$, we have
    \[
        \|b_{i,P}\|_{L^r_{B_1}}\leq 2^{1+\frac{1}{r}} \|f_i\mathbbm{1}_P\|_{L^r_{B_1}}\numberthis\label{singleBadInLp}
    \]
    and thus for every dyadic cube $Q$, by the disjointness of the $P$, we have
    \begin{equation*} \left\|\sum_{P\subset Q}b_{i,P}\right\|_{L^r_{B_1}}\lesssim\left(\sum_{P\subset Q}\|f_i\mathbbm{1}_P\|_{L^r_{B_1}}^r\right)^\frac{1}{r}\leq\left(\int_Q|f_i(x)|^r_{B_1}\intD x\right)^\frac{1}{r}\numberthis\label{badInLp}
    \end{equation*}
    and similarly 
    \[
        \left\|\sum_{P\subset Q}b^*_{i,P}\right\|_{L^r_{B_2^*}}\lesssim\left(\int_Q|g_i(x)|^r_{B_2^*}\intD x\right)^\frac{1}{r}.
    \]
    We decompose 
    \begin{align*}
        |\langle Tf_i&,g_i\rangle-\sum_{P\in\mathcal{W}}\langle S_Pf_i,g_i\rangle|\\&=|\langle Th_i,g_i\rangle- \sum_{P\in\mathcal{W}}\langle S_P\Big[\mathbbm{1}_P\fint_Pf_i\Big],g_i\rangle+\sum_{P\in\mathcal{W}}\langle(T-S_P)b_{i,P},g_i\rangle|\\&\leq |\langle Th_i,g_i\rangle|+| \sum_{P\in\mathcal{W}}\langle S_P\Big[\mathbbm{1}_P\fint_Pf_i\Big],g_i\rangle|+| \sum_{P\in\mathcal{W}}\langle(T-S_P)b_{i,P},g_i\rangle|\\&\eqqcolon |\langle Th_i,g_i\rangle|+|II|+|III|
        .
    \end{align*}
    Using the $L_{B_1}^{q,1}\to L^q_{B_2}$ boundedness of $T$ and \eqref{GoodIsRestrictedStrongType} with $r_1=q<\infty$ we get
    \begin{align*}
        |\langle Th_i,g_i\rangle|&\leq\|Th_i\|_{L^{q}_{B_2}}\|\mathbbm{1}_{3Q_0}g_i\|_{L^{q'}_{B_2^*}} \\
        &\leq A(q) \|h_i\|_{L^{q,1}_{B_2}}\|\mathbbm{1}_{3Q_0}g_i\|_{L^{q'}_{B_2^*}} \\&\lesssim_{d,q,\gamma}A(q) n^\frac{1}{p}\,|Q_0|\,\| f_i\|_{\textit{\L}^p(Q_0,B_1)}\| g_i\|_{\textit{\L}^{q'}(3Q_0,B_2^*)}\\&\leq A(q) n^\frac{1}{p}\,|3Q_0|\,\| f_i\|_{\textit{\L}^p(3Q_0,B_1)}\| g_i\|_{\textit{\L}^{q'}(3Q_0,B_2^*)}
        .
    \end{align*}
    The restricted strong type $(q,q)$ condition implies that $\|S_P\|_{L^{q,1}_{B_1}\to L^{q}_{B_2}}\leq A(q)$, and also 
    \[
        \left\|\mathbbm{1}_P\fint_Pf_i\right\|_{L^{q,1}_{B_1}}\lesssim_q\left|\fint_Pf_i\right|_{B_1}|P|^\frac{1}{q}\leq \| f_i\|_{\textit{\L}^p(P,B_1)}|P|^\frac{1}{q}.
    \]
    Together with the Hölder inequality, \eqref{arverageDilation} and the disjointness of the cubes $P$ we can estimate 
    \begin{align*}            |II|&=|\sum_{P\in\mathcal{W}}\langle S_P\Big[\mathbbm{1}_P\fint_Pf_i\Big],g_i\rangle|\\
    &\leq \sum_{P\in\mathcal{W}}\left\|S_P\Big[\mathbbm{1}_P\fint_Pf_i\Big]\right\|_{L^q_{B_2}}\|g_i\mathbbm{1}_{3P}\|_{L^{q'}_{B_2^*}}
    \\&\lesssim_{d,q} A(q)\sum_{P\in\mathcal{W}}\| f_i\|_{\textit{\L}^p(P,B_1)}|P|^\frac{1}{q}\| g_i\|_{\textit{\L}^{q'}(3P,B_2^*)}|P|^\frac{1}{q'}
    \\&\lesssim_{d,\gamma} A(q)n^{\frac{1}{p}+\frac{1}{q'}}\| f_i\|_{\textit{\L}^p(Q_0,B_1)}\| g_i\|_{\textit{\L}^{q'}(3Q_0,B_2^*)}\sum_{P\in\mathcal{W}}|P|
    \\&\leq A(q)n^{\frac{1}{p}+\frac{1}{q'}}\,|Q_0|\,\| f_i\|_{\textit{\L}^p(Q_0,B_1)}\| g_i\|_{\textit{\L}^{q'}(3Q_0,B_2^*)}\\&\leq A(q)n^{\frac{1}{p}+\frac{1}{q'}}\,|3Q_0|\,\| f_i\|_{\textit{\L}^p(3Q_0,B_1)}\| g_i\|_{\textit{\L}^{q'}(3Q_0,B_2^*)}.
    \end{align*}
    
    Now consider the third quantity $III= \sum_{P\in\mathcal{W}}\langle(T-S_P)b_{i,P},g_i\rangle$. We use $g_i=b_i^*+h_i^*$ and 
    \[
        (T-S_P)b_{i,P}=\sum_jT_jb_{i,P}-\sum_{j\leq L(P)}T_jb_{i,P}=\sum_{j>L(P)}T_jb_{i,P}
    \]
    to split $III=\sums{1}{i}{4}III_i$, where
    \begin{align*}
        III_1&=\langle \sum_{P\in\mathcal{W}}Tb_{i,P},h^*_i\rangle,\\
        III_2&=-\sum_{P\in\mathcal{W}}\langle S_Pb_{i,P},h^*_i\rangle,\\
        III_3&=\sum_{N_1\leq j\leq N_2}\sum_{\substack{P\in\mathcal{W}\\L(P)<j}}\sum_{\substack{P'\in\mathcal{W}\\L(P')\geq j}}\langle T_jb_{i,P},b^*_{i,P'}\rangle
    \end{align*}
and
\begin{align*}
    III_4&=\sum_{N_1\leq j\leq N_2}\sum_{\substack{P\in\mathcal{W}\\L(P)<j}}\sum_{\substack{P'\in\mathcal{W}\\L(P')<j}}\langle T_jb_{i,P} ,b^*_{i,P'}\rangle.
\end{align*}
Recall that for small enough $P$ the functions $f_i$ and $g_i$ will be constants on $P$, which implies that $b_{i,P}=b_{i,P}^*=0$. Thus there are no issues of convergence in the above decomposition.

To estimate $III_1$ we first note that the Hölder inequality of the scalar Lorentz spaces implies that
\[
    |III_1|\leq \int_{\R^d}\Big|T[\sum_{P\in\mathcal{W}}b_{i,P}](x)\Big|_{B_2}|h_i^*(x)|_{B_2^*}\intD x\leq\Big\|T[\sum_{P\in\mathcal{W}}b_{i,P}]\Big\|_{L^{p,\infty}_{B_2}}\|h^*_i\|_{L^{p',1}_{B_2^*}}.
\]
Thus the weak type condition  $(p,p)$ that $T$ maps boundedly from $L^p_{B_1}$ to $L^{p,\infty}_{B_2}$ together with  \eqref{GoodIsRestrictedStrongTypeStar} for $r_2=p'$ and \eqref{badInLp} for $r=p$ yields
\begin{align*}
    |III_1|&\leq \Big\|T[\sum_{P\in\mathcal{W}}b_{i,P}]\Big\|_{L^{p,\infty}_{B_2}}\|h^*_i\|_{L^{p',1}_{B_2^*}}    \\&\lesssim_{d,p,\gamma}A(p)\Big\|\sum_{P\in\mathcal{W}}b_{i,P}\Big\|_{L^{p}_{B_2}}n^\frac{1}{q'}|Q_0|^\frac{1}{p'}\| g_i\|_{\textit{\L}^{q'}(3Q_0,B_2^*)}
    \\&\lesssim A(p)n^\frac{1}{q'}\,|3Q_0|\,\| f_i\|_{\textit{\L}^p(3Q_0,B_1)}\| g_i\|_{\textit{\L}^{q'}(3Q_0,B_2^*)}.
\end{align*}
The weak type $(p,p)$ condition means that also $\|S_P\|_{L^{p}_{B_1}\to L^{p,\infty}_{B_2}}\leq A(p)$. Using this, \eqref{arverageDilation} and \eqref{goodIsInLinfty} we estimate
\begin{align*}
    |III_2|&\leq \sum_{P\in\mathcal{W}}\|S_Pb_{i,P}\|_{L^{p,\infty}_{B_2}}\|h_i^*\mathbbm{1}_{3P}\|_{L^{p',1}_{B_2^*}}
    \\&\leq A(p)\sum_{P\in\mathcal{W}}\|b_{i,P}\|_{L^{p}_{B_1}}\|h_i^*\|_{L^{\infty}_{B_2^*}}\|\mathbbm{1}_{3P}\|_{L^{p',1}_{B_2^*}}
    \\&\lesssim_{d,p} A(p)\sum_{P\in\mathcal{W}}|P|^\frac{1}{p}\| f_i\|_{\textit{\L}^p(P,B_1)}\|h_i^*\|_{L^{\infty}_{B_2^*}}|P|^\frac{1}{p'}
    \\&\lesssim_{d,\gamma} A(p) n^{\frac{1}{p}+\frac{1}{q'}} \,|3Q_0|\, \| f_i\|_{\textit{\L}^p(3Q_0,B_1)}\| g_i\|_{\textit{\L}^{q'}(3Q_0,B_2^*)}.
\end{align*}
In order to estimate the term $III_3$ we use the fact that for all $x\in\R^d$,
\[
    \begin{cases}\langle  T_jb_{i,P},b_{i,P'}^*\rangle\neq 0\\L(P)<j\leq L(P')\end{cases} \Longrightarrow\quad
    j\leq L(P')\leq L(P)+2<j+2.\numberthis\label{WWprimeclose}
\]
The above relation is (4.20) from \cite{beltran_multi-scale_2020} with $f_i$ and $g_i$ in place of $f_1$ and $f_2$ respectively, and their proof holds in our setting as well.

We use \eqref{WWprimeclose}, Hölder inequality, the single scale condition \[||T_j||_{L^p_{B_1}\to L^q_{B_2}}\leq 2^{-jd(\frac{1}{p}-\frac{1}{q})}A_\circ(p,q),\numberthis\label{singleScaleBoundedness}\] \eqref{singleBadInLp} and \eqref{arverageDilation} to estimate 
\begin{align*}
    |III_3|&\leq \sum_{N_1\leq j\leq N_2}\sum_{\substack{P,P'\in\mathcal{W}\\ j\leq L(P')\leq j+2\\ L(P')-2\leq L(P)\leq j}}|\langle T_jb_{i,P},b_{i,P'}^*\rangle|
    \\&\leq A_\circ(p,q) \sum_{N_1\leq j\leq N_2}2^{-jd(\frac{1}{p}-\frac{1}{q})}\sum_{\substack{P,P'\in\mathcal{W}\\ j\leq L(P')\leq j+2\\ L(P')-2\leq L(P)\leq j}}\|b_{i,P}\|_{L^p_{B_1}}\|b^*_{i,P'}\|_{L^{q'}_{B_2^*}}\\
    &\lesssim A_\circ(p,q) \sum_{N_1\leq j\leq N_2}2^{-jd(\frac{1}{p}-\frac{1}{q})}\\&\qquad\qquad\qquad\qquad\sum_{\substack{P,P'\in\mathcal{W}\\ j\leq L(P')\leq j+2\\ L(P')-2\leq L(P)\leq j}}\| f_i\|_{\textit{\L}^p(P,B_1)}|P|^\frac{1}{p}\| g_i\|_{\textit{\L}^{q'}(P,B_2^*)}|P'|^\frac{1}{q'}
    \\
    &\lesssim_{d,\gamma} n^{\frac{1}{p}+\frac{1}{q'}}A_\circ(p,q)\| f_i\|_{\textit{\L}^p(Q_0,B_1)}\| g_i\|_{\textit{\L}^{q'}(3Q_0,B_2^*)}\\ &\qquad\qquad\qquad\qquad\qquad\sum_{N_1\leq j\leq N_2}2^{-jd(\frac{1}{p}-\frac{1}{q})}\sum_{\substack{P,P'\in\mathcal{W}\\ j\leq L(P')\leq j+2\\ L(P')-2\leq L(P)\leq j}}|P|^\frac{1}{p}|P'|^\frac{1}{q'}.
\end{align*}
Since $2^{-jd\frac{1}{p}}\leq2^{-L(P)d\frac{1}{p}}=|P|^{-\frac{1}{p}}$ 
and
$2^{jd\frac{1}{q}} \leq 2^{L(P')d\frac{1}{q}}=|P'|^\frac{1}{q}$,
we get
\begin{align*}
    |III_3|\lesssim_{d,\gamma} n^{\frac{1}{p}+\frac{1}{q'}}A_\circ(p,q)\| f_i\|_{\textit{\L}^p(Q_0,B_1)}\| g_i&\|_{\textit{\L}^{q'}(3Q_0,B_2^*)}\\ &\sum_{N_1\leq j\leq N_2}\sum_{\substack{P'\in\mathcal{W}\\j\leq L(P')\leq j+2}} |P'|.
\end{align*}
Note that in the above double sum any cube $P'$ can appear at most three times and thus by disjointness of the $P'$, it follows that
\[
    |III_3|\lesssim_{d,\gamma}n^{\frac{1}{p}+\frac{1}{q'}}A_\circ(p,q)\,|Q_0|\,\| f_i\|_{\textit{\L}^p(3Q_0,B_1)}\| g_i\|_{\textit{\L}^{q'}(3Q_0,B_2^*)}.
\]

Finally, we consider the term
\begin{align*}
    III_4&=\sum_{N_1\leq j\leq N_2}\sum_{\substack{P,P'\in\mathcal{W}\\L(P)<j\\L(P')<j}}\langle T_jb_{i,P},b^*_{i,P'}\rangle.
\end{align*}
Let $\tilde\kappa>0$ such that
\[\tilde\kappa<\min\{\frac{1}{p},\frac{1}{q'},\kappa\}\]
and let $\ell$ be a positive integer such that\[\ell<\frac{100}{\tilde\kappa}\log_2(2+\frac{B}{A_\circ(p,q)})\leq \ell+1.\numberthis\label{elldef}\]
We split $\mathcal{V}_j=\left]-\infty,j\right[^2\cap\Z^2$ into three regions
\begin{align*}
    \mathcal{V}_{j,1}\coloneqq\{(L_1,L_2)\in \mathcal{V}_j\,&\colon\,j-\ell\leq L_1<j, j-\ell\leq L_2<j\},\\
    \mathcal{V}_{j,2}\coloneqq\{(L_1,L_2)&\in \mathcal{V}_j\setminus\mathcal{V}_{j,1}\,\colon\,L_1\leq L_2\}
\end{align*}
and 
\[\mathcal{V}_{j,3}\coloneqq\{(L_1,L_2)\in \mathcal{V}_j\setminus\mathcal{V}_{j,1}\,\colon\,L_1> L_2\}.\]
Now $III_4=\sums{i}{1}{3}IV_i$ where for $i=1,2,3,$
\[IV_i=\sum_{N_1\leq j\leq N_2}\sum_{\substack{P,P'\in\mathcal{W}\\(L(P),L(P'))\in\mathcal{V}_{j,i}}}\langle T_jb_{i,P},b^*_{i,P'}\rangle.\]

Let $\mathfrak R_j$ be the collection of dyadic subcubes of $Q_0$ of side length $2^j$. To
estimate $IV_1$ we tile $Q_0$ with such cubes and write
\[
    IV_1=\sum_{N_1\leq j\leq N_2}\sum_{R\in\mathfrak{R}_j}\langle\sum_{\substack{P\subset R\\j-\ell\leq L(P)<j}}T_jb_{i,P},\sum_{j-l\leq L(P')<j}b^*_{i,P'}\rangle.
\]
If $3R\cap P'\neq\emptyset$, then $P'$ intersects some dyadic neighbour of $R$ that we denote by $N(R)$. The cubes $N(R)$ and $P'$ are dyadic, which means that either $P'\subset N(R)$ or $N(R)\subset P'$. The latter option is impossible since $L(P')<j$. This shows that \[3R\cap P'\neq\emptyset\quad\Rightarrow \quad P'\subset 3R.\]
Thus by Hölder’s inequality and the single scale $(p, q)$ condition \eqref{singleScaleBoundedness} we now have
\begin{align*}
    |IV_1|&\leq A_\circ(p,q)\sum_{N_1\leq j\leq N_2}\sum_{R\in\mathfrak{R}_j}2^{-jd(\frac{1}{p}-\frac{1}{q})}\\&\qquad\qquad\qquad\qquad\qquad\bigg\|\sum_{\substack{P\subset R\\j-\ell\leq L(P)<j}}b_{i,P}\bigg\|_{L^p_{B_1}}\bigg\|\sum_{\substack{P'\subset3R\\j-\ell\leq L(P')<j}}b^*_{i,P'}\bigg\|_{L^{q'}_{B_2^*}}\\&= A_\circ(p,q)\sum_{N_1\leq j\leq N_2}\sum_{R\in\mathfrak{R}_j}|R|^{-(\frac{1}{p}-\frac{1}{q})}\\&\qquad\qquad\qquad\quad\bigg(\sum_{\substack{P\subset R\\j-\ell\leq L(P)<j}}\|b_{i,P}\|_{L^p_{B_1}}^p\bigg)^\frac{1}{p}\bigg(\sum_{\substack{P'\subset 3R\\j-\ell\leq L(P')<j}}\|b^*_{i,P'}\|_{L^{q'}_{B_2^*}}^{q'}\bigg)^\frac{1}{q'}.
\end{align*}
Using \eqref{singleBadInLp} and \eqref{arverageDilation} the above expression can be bounded by $C_{d,\gamma}A_\circ(p,q)$ times
\begin{align*}
    &n^{\frac{1}{p}+\frac{1}{q'}}
    \sum_{N_1\leq j\leq N_2}\| f_i\|_{\textit{\L}^p(Q_0,B_1)}\| g_i\|_{\textit{\L}^{q'}(3Q_0,B_2^*)}\\&\qquad\qquad\qquad\qquad\sum_{R\in\mathfrak{R}_j}|R|^{-(\frac{1}{p}-\frac{1}{q})}\bigg(\sum_{\substack{P\subset R\\j-\ell\leq L(P)<j}}|P|\bigg)^\frac{1}{p}\bigg(\sum_{\substack{P'\subset 3R\\j-\ell\leq L(P')<j}}|P'|\bigg)^\frac{1}{q'}
    \\&\leq n^{\frac{1}{p}+\frac{1}{q'}}
    \sum_{N_1\leq j\leq N_2}\| f_i\|_{\textit{\L}^p(Q_0,B_1)}\| g_i\|_{\textit{\L}^{q'}(3Q_0,B_2^*)}\\&\qquad\qquad\qquad\qquad\sum_{R\in\mathfrak{R}_j}|R|^{-(\frac{1}{p}-\frac{1}{q})}\bigg(\sum_{\substack{P\subset 3R\\j-\ell\leq L(P)<j}}|P|\bigg)^{\frac{1}{p}-\frac{1}{q}}\bigg(\sum_{\substack{P\subset 3R\\j-\ell\leq L(P)<j}}|P|\bigg).
\end{align*}
Using $p\leq q$ and the disjointness of the $P$ we see that the last expression can be dominated by a dimensional constant times 
\[
    n^{\frac{1}{p}+\frac{1}{q'}}\| f_i\|_{\textit{\L}^p(Q_0,B_1)}\| g_i\|_{\textit{\L}^{q'}(3Q_0,B_2^*)}\sum_{N_1\leq j\leq N_2}\sum_{R\in\mathfrak{R}_j}\sum_{\substack{P\subset 3R\\j-\ell\leq L(P)<j}}|P|.
\]
In the above triple sum any cube $P\in\mathcal{W}$ may appear at most $3^d\ell$ times so we can conclude, using \eqref{elldef}, that
\begin{align*}
    |IV_1|&\lesssim_{d,\gamma} A_\circ(p,q)n^{\frac{1}{p}+\frac{1}{q'}}\ell\,|Q_0|\,\| f_i\|_{\textit{\L}^p(Q_0,B_1)}\| g_i\|_{\textit{\L}^{q'}(3Q_0,B_2^*)}
    \\&\lesssim_{\kappa,p,q}A_\circ(p,q)\log(2+\frac{B}{A_\circ(p,q)})n^{\frac{1}{p}+\frac{1}{q'}}\,|Q_0|\,\| f_i\|_{\textit{\L}^p(Q_0,B_1)}\| g_i\|_{\textit{\L}^{q'}(3Q_0,B_2^*)}.
\end{align*}

For the last two terms we claim that
\[|IV_2|+|IV_3|\lesssim_{d,p,q,\gamma,\kappa}A_\circ(p,q)\,n^{\frac{1}{p}+\frac{1}{q'}}\,|Q_0|\,\| f_i\|_{\textit{\L}^p(Q_0,B_1)}\| g_i\|_{\textit{\L}^{q'}(3Q_0,B_2^*)}.\]
Recall that $\mathbb{E}_kf\coloneqq\sum_{Q}\mathbbm{1}_Q\fint_Qf$, where the sum is taken over dyadic cubes $Q$ with side length $2^{-k}$. Note that with this notation we have
$b_{i,P}=(I-\mathbb{E}_{-L(P)})\mathbbm{1}_P f_i$ and $b^*_{i,P}=(I-\mathbb{E}_{-L(P)})\mathbbm{1}_P g_i$ for $i=1,\dots,n$.
Lemma \ref{BRScor} gives that
\[
    \|T_j(I-\mathbb{E}_{s_1-j})\|_{L^p_{B_1}\to L^q_{B_2}}\lesssim_{p,q,\kappa}B2^{-jd(\frac{1}{p}-\frac{1}{q})}2^{-\tilde\kappa s_1}\numberthis\label{condExpBound}
\]
and
\[
    \|T_j^*(I-\mathbb{E}_{s_2-j})\|_{L^{q'}_{B_2^*}\to L^{p'}_{B_1^*}}\lesssim_{p,q,\kappa}B2^{-jd(\frac{1}{p}-\frac{1}{q})}2^{-\tilde\kappa s_2}.\numberthis\label{adjCondExpBound}
\]

In the $IV_2$ term, we have $L(P)\leq L(P')$ and $L(P)<j-\ell$.
Thus with the same $\mathfrak R_j$ as before we write
\begin{align*}
    IV_2&=\sum_{N_1\leq j\leq N_2}\sum_{\substack{P,P'\in\mathcal{W}\\L(P)\leq L(P')\\L(P)<j-\ell}}\langle T_jb_{i,P},b^*_{i,P'}\rangle\\&=\sum_{N_1\leq j\leq N_2}\sum_{R\in\mathfrak{R}_j}\sums{s_2}{1}{\infty}\sums{s_1}{\max\{s_2,\ell+1\}}{\infty}\langle \sum_{\substack{P\subset R\\L(P)=j-s_1}}T_jb_{i,P},\sum_{\substack{P'\subset 3R\\L(P')=j-s_2}}b^*_{i,P'}\rangle.
\end{align*}
 
Note that for $L(P)=j-s_1$, we have $b_{i,P}=(I-\mathbb E_{s_1-j})\mathbbm 1_Pf_i
$. By Hölder's inequality and \eqref{condExpBound} we get for $R\in\mathfrak R_j$,
\begin{align*}
   |\langle \sum_{\substack{W\subset R\\L(P)=j-s_1}}T_j&b_{i,P},\sum_{\substack{P'\subset 3R\\L(P')=j-s_2}}b^*_{i,P'}\rangle|\\&\leq  \|\sum_{\substack{P\subset R\\L(P)=j-s_1}}T_jb_{i,P}\Big\|_{L^q_{B_2}}\Big\|\sum_{\substack{P'\subset 3R\\L(P')=j-s_2}}b^*_{i,P'}\Big\|_{L^{q'}_{B_2^*}}\\&\leq \Big\|T_j(I-\mathbb E_{s_1-j})\big[\sum_{\substack{P\subset R\\L(P)=j-s_1}}\mathbbm1_Pf_i\big]\Big\|_{L^q_{B_2}}\Big\|\sum_{\substack{P'\subset 3R\\L(P')=j-s_2}}b_{i,P'}^*\Big\|_{L^{q'}_{B_2^*}}
    \\&\lesssim_{p,q,\kappa} B2^{-\tilde\kappa s_1}|R|^{-(\frac{1}{p}-\frac{1}{q})}\Big(\sum_{\substack{P\subset R\\L(P)=j-s_1}}\|f_i\mathbbm 1_P\|_{L^p_{B_1}}^p\Big)^\frac{1}{p}\\&\qquad\qquad\qquad\qquad\qquad\qquad\qquad\qquad\Big(\sum_{\substack{P'\subset 3R\\L(P')=j-s_2}}\|b_{i,P'}^*\|_{L^{q'}_{B_2^*}}^{q'}\Big)^\frac{1}{q'}.
\end{align*}
Thus interchanging the $j$-sum and the $(s_1, s_2)$-sums yields
\begin{align*}
    |IV_2|\lesssim_{p,q,\kappa} &\sums{s_2}{1}{\infty}\sums{s_1}{\max\{s_2,\ell+1\}}{\infty}B2^{-\tilde\kappa s_1}\sums{j}{N_1}{N_2}\sum_{R\in\mathfrak R_j}|R|^{-(\frac{1}{p}-\frac{1}{q})}\\
    &\qquad\qquad\qquad \Big(\sum_{\substack{P\subset R\\L(P)=j-s_1}}\|f_i\mathbbm 1_P\|_{L^p_{B_1}}^p\Big)^\frac{1}{p}\Big(\sum_{\substack{P'\subset 3R\\L(P')=j-s_2}}\|b_{i,P'}^*\|_{L^{q'}_{B_2^*}}^{q'}\Big)^\frac{1}{q'}
\end{align*}
and the right-hand side is bounded by $C_{d,\gamma}$ times
\[
    n^{\frac{1}{p}+\frac{1}{q'}}\| f_i\|_{\textit{\L}^p(3Q_0,B_1)}\| g_i\|_{\textit{\L}^{q'}(3Q_0,B_2^*)}\sums{s_2}{1}{\infty}\sums{s_1}{\max\{s_2,\ell+1\}}{\infty}B2^{-\tilde\kappa s_1}\sums{j}{N_1}{N_2}\sum_{R\in\mathfrak R_{j}}\,\Gamma(R,j),
\]
where
\[\Gamma(R,j)\coloneqq |R|^{-(\frac{1}{p}-\frac{1}{q})}\Big(\sum_{\substack{P\subset R\\L(P)=j-s_1}}|P|\Big)^\frac{1}{p}\Big(\sum_{\substack{P'\subset 3R\\L(P')=j-s_2}}|P'|\Big)^\frac{1}{q'}.\]
Using $p\leq q$ we get a crude bound
\begin{align*}
    \Gamma(R,j)&\leq |R|^{-(\frac{1}{p}-\frac{1}{q})}\Big(\sums{\nu}{j-s_1}{j-s_2}\sum_{\substack{P\subset3R\\L(P)=\nu}}|P|\Big)^{\frac{1}{p}+1-\frac{1}{q}}
    \\&\leq |R|^{-(\frac{1}{p}-\frac{1}{q})}|3R|^{\frac{1}{p}-\frac{1}{q}}\sums{\nu}{j-s_1}{j-s_2}\sum_{\substack{P\subset3R\\L(P)=\nu}}|P|
    \\&=3^{d(\frac{1}{p}-\frac{1}{q})}\sum_{\substack{P\subset3R\\j-s_1\leq L(P)\leq j-s_2}}|P|.
\end{align*}
For fixed $P \in \mathcal{W}$ consider the set of all pairs $(R, j)$ such that $R \in \mathfrak R_{j}, P \subset 3R$ and $j-s_1\leq
L(P) \leq j - s_2$, and observe that the cardinality of this set is bounded above by $3^d(s_1 - s_2 + 1)$. Thus any $P\in\mathcal{W}$ can appear in the sum
\[
    \sums{j}{N_1}{N_2}\sum_{R\in\mathfrak R_{j}}\sum_{\substack{P\subset3R\\j-s_1\leq L(P)\leq j-s_2}}|P|
\]
at most $3^d(s_1 - s_2 + 1)$ times. Combining this
with the above estimates we obtain the bound
\begin{align*}|IV_2|\lesssim_{d,p,g,\gamma,\kappa}\| f_i\|_{\textit{\L}^p(3Q_0,B_1)}\| g_i\|_{\textit{\L}^{q'}(3Q_0,B_2^*)}&|Q_0|\\\sums{s_2}{1}{\infty}&\sums{s_1}{\max\{s_2,\ell+1\}}{\infty}B2^{-\tilde\kappa s_1}(s_1-s_2+1)\end{align*}
and for the double sum we have
\begin{align*}
     \sums{s_2}{1}{\ell}\sums{s_1}{\ell+1}{\infty}B2^{-\tilde\kappa s_1}(s_1-s_2+1)&+\sums{s_2}{\ell+1}{\infty}\sums{s_1}{s_2}{\infty}B2^{-\tilde\kappa s_1}(s_1-s_2+1)\\
     &\lesssim_{\kappa,p,q} \sums{s_2}{1}{\ell}B2^{-\tilde\kappa\ell}\ell + \sums{s_2}{\ell}{\infty}B2^{-\tilde\kappa s_2}s_2
     \\&\lesssim_{\kappa,p,q} B2^{-\tilde\kappa\ell}\ell^2\lesssim_{\kappa,p,q} B2^{-\frac{\tilde\kappa\ell}{2}} \lesssim_{\kappa,p,q}A_\circ(p,q).
\end{align*}
Above we used the fact that 
$\sums{k}{1}{m}kx^k=\frac{mx^{m+2}+x-mx^{m+1}-x^{m+1}}{(1-x)^2}$ holds for $m\in\N$ and $x\in\R\setminus\{1\}$. In last line we also used \eqref{elldef} and the fact that $\log y\lesssim_ay^a$ for $a,y>0$.

In the last term $IV_3$ we have $L(P)>L(P')$ and $L(P')<j-\ell$. Thus we may write 
\begin{align*}
    IV_3&=\sum_{N_1\leq j\leq N_2}\sum_{R\in\mathfrak{R}_j}\sums{s_1}{1}{\infty}\sums{s_2}{\max\{s_1+1,\ell+1\}}{\infty}\langle \sum_{\substack{P\subset R\\L(P)=j-s_1}}T_jb_{i,P},\sum_{\substack{P'\subset 3R\\L(P')=j-s_2}}b^*_{i,P'}\rangle\\&=\sum_{N_1\leq j\leq N_2}\sum_{R\in\mathfrak{R}_j}\sums{s_1}{1}{\infty}\sums{s_2}{\max\{s_1+1,\ell+1\}}{\infty}\langle \sum_{\substack{P\subset R\\L(P)=j-s_1}}b_{i,P},\sum_{\substack{P'\subset 3R\\L(P')=j-s_2}}T_j^*b^*_{i,P'}\rangle.
\end{align*}
For $R\in\mathfrak R_j$, by Hölder's inequality and \eqref{adjCondExpBound} we get
\begin{align*}
    &|\langle\sum_{\substack{P\subset R\\L(P)=j-s_1}}b_{i,P},\sum_{\substack{P'\subset 3R\\L(P')=j-s_2}}T_j^*b^*_{i,P'}\rangle|\\
    &\lesssim_{p,q,\kappa} B2^{-\tilde\kappa s_2}|R|^{-(\frac{1}{p}-\frac{1}{q})}\Big(\sum_{\substack{P\subset R\\L(P)=j-s_1}}\|b_{i,P}\|_{L^p_{B_1}}^p\Big)^\frac{1}{p}\Big(\sum_{\substack{P'\subset 3R\\L(P)=j-s_2}}\|g_i\mathbbm 1_{P'}\|_{L^{q'}_{B_2^*}}^{q'}\Big)^\frac{1}{q'}
\end{align*}
and from here on the argument is analogous to the treatment of the term
$IV_2$.
\end{proof}

\newpage


\end{document}